\documentclass[12pt,leqno]{article} 
\usepackage{bezier,amsmath,stmaryrd, amssymb,mathtools,lscape,amsthm, fontenc, pdflscape, oldgerm, nicefrac, array,upgreek}
\usepackage[paper=a4paper,left=25mm,right=25mm,top=25mm,bottom=25mm]{geometry}
\usepackage{rotating}
\usepackage[pdfencoding=auto, psdextra,pdfborder={0 0 0}]{hyperref}
\usepackage[nameinlink]{cleveref}
\input xy
\xyoption{all}
\renewcommand{\arraystretch}{1.2}
\newcolumntype{L}{>{$}l<{$}} 

\newcommand{\scm}{\scriptstyle}
\newcommand{\ti}{\times}

\newcommand{\Z}{\mathbf Z}
\newcommand{\Q}{\mathbf Q}

\renewcommand{\phi}{\varphi}

\newcommand{\F}{\mathbf F}

\newcommand{\YY}{\mathcal{Y}}
\newcommand{\XX}{\mathcal{X}}
\newcommand{\ab}{\hspace*{0.4mm}}

\DeclareMathOperator{\SSS}{S}

\newcommand{\sd}{\SSS_3}

\newcommand{\al}{\alpha}
\newcommand{\be}{\beta}

\newcommand{\ga}{\gamma}

\newcommand{\ttau}{\tilde{\tau}}

\newcommand{\eps}{\varepsilon}

\newcommand{\oeta}{\overline{\eta}}
\newcommand{\oxi}{\overline{\xi}}

\newcommand{\bp}{\begin{picture}}
\newcommand{\ep}{\end{picture}}
\renewcommand{\scm}{\scriptstyle}

\newcommand{\rxisoa}[1]{\xrightarrow[\bp(25,0)\put(0,15){$\scm\sim$}\ep]{#1}}

\DeclareMathOperator{\B}{B}
\DeclareMathOperator{\Bi}{Biset}
\DeclareMathOperator{\Mack}{Mack}

\DeclareMathOperator{\UU}{U}

\DeclareMathOperator{\kernel}{kern}

\DeclareMathOperator{\Aut}{Aut}
\DeclareMathOperator{\e}{e}

\DeclareMathOperator{\rk}{rk}

\DeclareMathOperator{\idd}{id}
\DeclareMathOperator{\HH}{H}
\DeclareMathOperator{\CC}{C}
\DeclareMathOperator{\A}{A}
\DeclareMathOperator{\KK}{K}
\DeclareMathOperator{\Rep}{Rep}
\newcommand{\ul}{\underline}
\newcommand{\tiG}{\underset{G}{\ti}}

\newcommand{\cdotG}{\underset{G}{\cdot}}
\newcommand{\enger}{\setlength{\arraycolsep}{3pt}
	\renewcommand{\arraystretch}{1.0} }
\newcommand{\weiter}{\setlength{\arraycolsep}{3pt}
	\renewcommand{\arraystretch}{1.3} }

\newcommand\Tstrut{\rule{0pt}{2.6ex}}       
\newcommand\Bstrut{\rule[-0.9ex]{0pt}{0pt}}

\newcommand{\rsmatzd}[6]{\enger
	\left(
	\begin{array}{rrr}
		\small #1 & \small #2 & \small #3 \\
		\small #4 & \small #5 & \small #6 \\
	\end{array}
	\right) 
	\weiter }
\newtheorem{Lemma}{Lemma}
\newtheorem{Lemma*}{Lemma A}

\newtheorem{Remark}[Lemma]{Remark}
\newtheorem{Notation}[Lemma]{Notation}
\newtheorem{Remark*}[Lemma]{}
\newtheorem{Corollary}[Lemma]{Corollary}
\newtheorem{Theorem}[Lemma]{Theorem}
\newtheorem{Proposition}[Lemma]{Proposition}

\begin{document}
\title{The integral double Burnside ring of the\\ symmetric group $\sd$}
\author{Nora Krau{\ss}}
\date{}
\maketitle

 The double Burnside $R$-algebra $\B_R(G,G)$ of a finite group $G$ with coefficients in a commutative ring $R$ has been introduced by S.\ Bouc. It is $R$-linearly generated by finite $(G,G)$-bisets, modulo a relation identifying disjoint union and sum. Its multiplication is induced by the tensor product. 
 In his thesis at NUI Galway, B.\ Masterson described $\B_\Q(\sd,\sd)$ as a subalgebra of $\Q^{8\ti 8}$. We give a variant of this description and continue to describe $\B_R(\sd,\sd)$ for $R\in\{\Z,\Z_{(2)},\F_2,\Z_{(3)},\F_3\}$ via congruences as suborders of certain $R$-orders respectively via path algebras over $R$.

\renewcommand{\thefootnote}{\fnsymbol{footnote}}
\renewcommand{\thefootnote}{\arabic{footnote}}

\section{Introduction}
\subsection{Groups}
Groups describe symmetries of objects. That is to say, any mathematical object $X$ has a symmetry group, called automorphism group $\Aut(X)$, consisting of isomorphisms from $X$ to $X$. For instance, for a natural number $n$, the set $\{1,2,\dots,n\}$ has as automorphism group the symmetric group $\Aut(\{1,2,\dots,n\})=\SSS_n$. This group consists of all bijections from $\{1,2,\dots,n\}$ to itself.
For example, we obtain
\[\begin{array}{rcl}\sd &=& \left\{\rsmatzd{1}{2}{3}{1}{2}{3}, \rsmatzd{1}{2}{3}{2}{1}{3}, \rsmatzd{1}{2}{3}{3}{2}{1}, \rsmatzd{1}{2}{3}{1}{3}{2}, \rsmatzd{1}{2}{3}{2}{3}{1},\rsmatzd{1}{2}{3}{3}{1}{2} \right\}\\\\
&=& \left\{\idd, (1,2),(1,3),(2,3),(1,2,3),(1,3,2) \right\}~. \end{array}\]
In the first row, $\rsmatzd{1}{2}{3}{a}{b}{c}$ is the map sending $1\mapsto a$, $2\mapsto b$, $3\mapsto c$.\\
In the second row, we have used the cycle notation, e.g.\ $\rsmatzd{1}{2}{3}{3}{1}{2}=(1,3,2)$, the latter meaning {\xymatrix@C-1.8em@R-1.4em{
		1\ar@{|-{>}}@/^1pc/[rr]& & 3\ar@{|-{>}}@/^1pc/[dl] \\
		& 2\ar@{|-{>}}@/^1pc/[ul] & 
}}.\\
We multiply by composition, e.g.\ $(1,2)\bullet(1,3)=(1,2,3)$.\\
By a theorem of Cayley, any finite group is isomorphic to a subgroup of $\SSS_n$ for some $n$.

\subsection{The Biset category and biset functors}\label{biset}
Suppose given finite groups $H$ and~$G$. An \textit{$(H,G)$-biset $X$} is a finite set $X$ together with a multiplication with elements of $H$ on the left and a multiplication with elements of $G$ on the right that commute with each other, i.e.\ \[(h\cdot x)\cdot g = h\cdot (x\cdot g)=:h\cdot x\cdot g \] for $h\in H$, $g\in G$ and $x\in X$.

As a first example, $M_1:=\SSS_3$ is a $(\SSS_2,\SSS_3)$-biset via multiplication in $\sd$. So for  $h\in \SSS_2=\{\idd, (1,2)\}$, $g\in \sd$ and $x\in M_1$ we let $h\cdot x\cdot g:=h\bullet x\bullet g$.

As a second example, consider the cyclic group $\CC_3=\{\idd,(1,2,3),(1,3,2)\}$ and the group isomorphism $\al:\CC_3\rightarrow\CC_3$, $x\mapsto x^2$. Then the set $M_2:=\CC_3$ is a $(\CC_3,\CC_3)$-biset, on the left via multiplication, on the right via application of $\al$ and then multiplication. E.g.\[\begin{array}{rclcl}(1,2,3)\cdot(1,3,2)\cdot (1,3,2) &=& (1,2,3)\bullet(1,3,2)\bullet\al((1,3,2))&&\\&=&(1,2,3)\bullet(1,3,2)\bullet(1,2,3)&=&(1,2,3)~.\end{array}\]

Suppose given a commutative ring $R$. S.\ Bouc introduced the {\it biset category $\Bi_{R}$}, see \cite[\S 3.1]{Bouc}, see also the historical comments in \cite[\S 1.4]{Bouc}. As objects, the category $\Bi_R$ has finite groups. The $R$-module of morphisms between two finite groups $H$ and $G$ is given by the double Burnside \mbox{$R$-module} $\Bi_{R}(H,G)=\B_R(H,G)$, which is $R$-linearly generated by finite $(H,G)$-bisets, modulo a relation identifying disjoint union and sum. In particular, each $(H,G)$-biset $M$ yields a morphism $H\xrightarrow{[M]}G$ in $\Bi_{R}$. Composition of morphisms in $\Bi_R$ is given by a tensor product operation on bisets that is similar to the tensor product of bimodules. Given an $(H,G)$-biset $M$ and an $(G,K)$-biset $N$, we write $M\tiG N$ for their tensor product, which is an $(H,K)$-biset. So in $\Bi_{R}$, we have the commutative triangle \[
{\xymatrix@C-0.5em@R-0.5em{
		&H\ar[rr]^-{[M \tiG N]}\ar[dr]_-{[M]} & & K & & \\                                                 
		& & G\ar[ur]_-{[N]} & &  \\
}}.\] The category $\Bi_{R}$ may roughly be imagined by a picture like this.
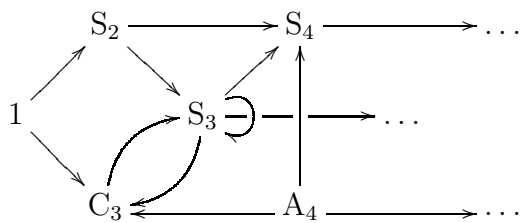
\[
{\xymatrix@C-0.5em@R-0.5em{
		& \SSS_2\ar[rr]\ar[dr] & & \SSS_4\ar[rr] & & \dots\\                                                 
		1\ar[ur]\ar[dr]& & \sd\ar@(ur,dr)\ar@/^1pc/[dl]\ar[ur] \ar[rr]|(.26)\hole|(.49)\hole & & \dots \\
		& \CC_3\ar@/^1pc/[ur] & & \A_4\ar[]!<-0.5ex,-0.5ex>;[ll]!<-0.5ex,-0.5ex> \ar[uu]\ar[]!<-0.5ex,-0.5ex>;[rr]!<-0.5ex,-0.5ex> & & \dots
}}\] Here, $\A_4$ is the alternating group on $4$ elements. Each biset yields an arrow, and so does each $R$-linear combination of bisets. Of course, there are many more objects in $\Bi_{R}$ -- each finite group is an object there -- and many more arrows between them that are not in our picture.

\subsection{Biset functors}
Let $\mathcal{X}$ and $\mathcal{Y}$ be classes of finite groups closed under forming subgroups, factor groups and extensions.
Following Bouc \cite[\S 3.4.1]{Bouc2}, we say that an $(H,G)$-biset $M$ is {\it $(\XX,\YY)$-free} if for each $m\in M$ the left stabilizer of $m$ in $H$ is in $\XX$ and the right stabilizer of $m$ in~$G$ is in $\YY$. We have the subcategory $\Bi_{R}^{\XX,\YY}$ of $\Bi_R$\ab: As objects, it has finite groups. The $R$-module of morphisms in $\Bi_{R}^{\XX,\YY}$ between two finite groups $H$ and $G$ is given by the submodule of $\B_R(H,G)$ generated by the images of $(\XX,\YY)$-free $(H,G)$-bisets, cf.\ \cite[Lemme 4]{Bouc2}. 

Certain classical theories may now be formulated as contravariant functors from $\Bi_{R}^{\XX,\YY}$ to the category of $R$-modules, called {\it biset functors} over $R$. 

Consider a prime number $p$. Let $\XX$ be the class of all finite groups. Let $\YY$ be the class of finite groups whose orders are not divisible by $p$. Then e.g.\ the $(\SSS_2,\sd)$-biset $M_1$ and the $(\CC_3,\CC_3)$-biset $M_2$ from \S\ref{biset} yield morphisms in $\Bi_{\Z}^{\XX,\YY}$.

Suppose given an object of $\Bi_{\Z}^{\XX,\YY}$, i.e.\ a finite group $G$. Let $$\F_{p}=\Z/p\Z=\{0,\dots,p-1\}~,$$ where we agree to calculate modulo $p$. An $\F_{p}$-representation of $G$ is a finite dimensional $\F_{p}$-vectorspace $V$, together with a left multiplication with elements of $G$. Such a representation is called simple if it does not have a nontrivial subrepresentation. Each representation has a sequence of subrepresentations with simple steps, called composition factors. 

Let $\Rep_{\F_{p}}(G)$ be the free abelian group on the set of isoclasses of simple representations. Each $\F_p$-representation $V$ of $G$ yields an element $[V]$ in $\Rep_{\F_{p}}(G)$, namely the formal sum of its composition factors. Given finite groups $H$ and $G$ and an $(H,G)$-biset $M$, we obtain the map \[\begin{array}{rcl}
\Rep_{\F_{p}}(G) &\xrightarrow{\Rep_{\F_{p}}([M])} & \Rep_{\F_{p}}(H)\\
{[V]} &\mapsto & [\F_{p}M\underset{\F_{p}G}{\otimes} V]~,
\end{array}\]using the usual tensor product over rings.

These constructions furnish a contravariant $\Z$-linear functor $\Rep_{\F_{p}}$ from $\Bi_{\Z}^{\XX,\YY}$ to the category of $\Z$-modules, i.e.\ to the category of abelian groups. In particular, using the bisets $M_1$ and $M_2$ from \S\ref{biset}, we obtain the maps 
\[\begin{array}{rcl}
\Rep_{\F_{p}}(\sd) &\xrightarrow{\Rep_{\F_{p}}([M_1])} & \Rep_{\F_{p}}(\SSS_2)\\
{[V]} &\mapsto & [\text{restriction of $V$ to $\SSS_2$}]
\end{array}\] and \[\begin{array}{rcl}
\Rep_{\F_{p}}(\CC_3) &\xrightarrow{\Rep_{\F_{p}}([M_2])} & \Rep_{\F_{p}}(\CC_3)\\
{[V]} &\mapsto & [\text{twist of $V$ with $\al$}]~.
\end{array}\]
Note that, if $p\leq n$, even the simple $\F_p$-representations of $\SSS_n$ are not entirely known: One knows a construction, due to James \cite{James}, but one does not know their $\F_{p}$-dimensions.
Biset functors do not directly aim to solve this problem, but at any rate they are a tool to work with these representations.

\subsection{Globally-defined Mackey functors}
There is an equivalence of categories between the category of biset functors over $R$ and the category of {\it globally-defined Mackey functors $\Mack_{R}^{\XX,\YY}$} \cite[\S 8]{Webb}. Here, a globally-defined Mackey functor, with respect to $\XX$ and $\YY$, maps groups to $R$-modules and each group morphism $\al$ covariantly to an $R$-module morphism $\al_\star$\ab, provided $\kernel(\al)\in\YY$, and contravariantly to $\al^\star$\ab, provided $\kernel(\al)\in\XX$. It is required that these morphisms satisfy a list of compatibilities, amongst which a Mackey formula, see e.g.\ \cite[\S 8]{Webb}. By that equivalence, these requirements on a Mackey functor can now be viewed as properties that result from being a contravariant functor from $\Bi_{R}^{\XX,\YY}$ to $R$-Mod. 
\subsection{Further examples}
We list two examples of biset functors, \cite[\S 8]{Webb}.
\begin{itemize}
	\item Let $\XX=\{1\}$ and let $\YY$ consist of all finite groups. Let $n\geq 0$. Consider the biset functor $\Bi_{\Z}^{\XX,\YY}\rightarrow \Z$-Mod that maps a finite group $G$ to the algebraic K-theory $\KK_n(\Z G)$ of $\Z G$. 
	\item Let $\XX$ consist of all finite groups and let $\YY=\{1\}$. Let $n\geq 0$. Consider the biset functor $\Bi_{R}^{\XX,\YY}\rightarrow R$-Mod that maps a finite group $G$ to the cohomology  $\HH^n(G, R)$ of $G$ with trivial coefficients.\end{itemize}
For some more examples, see \cite[\S 8]{Webb}. The example of the classical Burnside ring, depending on a group $G$, is also explained in \cite[\S 6.1]{Bouc3}.

\subsection{The double Burnside algebra}
Suppose given a finite group $G$, i.e\ an object of  $\Bi_R$\ab. Its endomorphism ring  $\B_R(G,G)$ in the category $\Bi_{R}$ is called {\it double Burnside algebra} of $G$.  

The isomorphism classes of finite transitive $(G,G)$-bisets form an $R$-linear basis of $\B_R(G,G)$. In particular, if we choose a system $\mathcal{L}_{G\ti G}$ of representatives for the conjugacy classes of subgroups of $G\ti G$,  we have the $R$-linear basis $([(G\ti G)/U] : U\in \mathcal{L}_{G\ti G})$.

If $G$ is cyclic and if $R$ is a field in which $|G|$ and $\upvarphi(|G|)$ are invertible, where $\upvarphi$ denotes  Euler's totient function, then the double Burnside algebra $\B_R(G,G)$ is semisimple. This is shown in \cite[Theorem 8.11, Remark 8.12(a)]{Boltje-Danz}.   

In case of $G=\sd$, we have $22$ conjugacy classes of subgroups of $\sd\ti\sd$ and thus $\rk_R(\B_R(\sd,\sd))=22$.
The double Burnside $\Q$-algebra $\B_\Q(\sd,\sd)$ has been described by B. Masterson \cite[\S 8]{Masterson} and then by B. Masterson and G. Pfeiffer \cite[\S7]{Pfeiffer}. We describe $\B_\Q(\sd,\sd)$ independently, using a direct Magma-supported calculation \cite{Magma}, with the aim of being able to pass from $\B_\Q(\sd,\sd)$ to $\B_\Z(\sd,\sd)$ in the sequel.

In order to do that, we first restate some preliminaries on bisets and the double Burnside ring in \S\ref{prelim} and construct a $\Z$-linear basis of $\B_\Z(\sd,\sd)$ in \S\ref{Basis}.

In \S\ref{rational} we tackle the problem that the double Burnside $\Q$-algebra $\B_\Q(\sd,\sd)$  is not semisimple \cite[Proposition 6.1.5]{Bouc}, thus not isomorphic to a direct product of matrix rings. 
As a substitute, we use a suitable isomorphic copy $A$ of $\B_\Q(\sd,\sd)$. We obtain this copy using a Peirce decomposition of $\B_\Q(\sd,\sd)$. In addition, we give a description of $\B_\Q(\sd,\sd)$ as path algebra modulo relations.

The next step, in \S\ref{int}, is to pass from $\B_\Q(\sd,\sd)$ to $\B_\Z(\sd,\sd)$. We find a $\Z$-order $A_\Z$ inside $A$ such that $A_\Z$ contains an isomorphic copy of $\B_\Z(\sd,\sd)$, which we describe via congruences, cf.\ \Cref{gamma}, \Cref{Proplamb}.
\[
\begin{array}{l}
\xymatrix@C=24mm{
	\B_\Q(\sd,\sd)\ar[r]^{\sim}  & A \\
	\overset{\phantom{.}}{\B_\Z(\sd,\sd)}	\ar@{^{(}->}[u]\ar[r]^{\text{injective}} &\overset{\phantom{.}}{A_\Z}\ar@{^{(}->}[u]
}\\
\end{array}
\] 

We calculate a path algebra for $\B_{\Z_{(2)}}(\sd,\sd)$, cf.\ \Cref{pfadsdZ2}. We deduce that $\B_{\F_{2}}(\sd,\sd)$ is Morita equivalent to the path algebra\[\F_2\left[{\xymatrix@C-0.5em@R-0.5em{
		\tilde \e_3\ar@/^1pc/[rr]^-{\hspace{-0.7mm}\tilde \tau_2} & &\tilde \e_5\ar@(ur,dr)^-{\hspace{-0.7mm}\tilde\tau_7}\ar@/^1pc/[ll]^-{\hspace{-0.7mm}\tilde \tau_1}\ar@/_2pc/[rr]_-{\hspace{-0.7mm}\tilde \tau_3} & &  \tilde \e_4\ar@/_2pc/[ll]_-{\hspace{-0.7mm}\tilde\tau_4}}}\right] /\left(\begin{array}{ccccccccccc}
\ttau_2\ttau_1   &,& \ttau_2\ttau_3  &,& \ttau_2\ttau_7,  \\
\ttau_4\ttau_1&,& \ttau_4\ttau_3  &,& \ttau_4\ttau_7, \\
\ttau_7\ttau_1&,& \ttau_7\ttau_3&,&  \ttau_7^2-\ttau_1\ttau_2~ \\
\end{array}\right) ~, \]cf.\ \Cref{pfadsdF2}.

We calculate a path algebra for $\B_{\Z_{(3)}}(\sd,\sd)$, cf.\ \Cref{pfadsdZ3}. We deduce that $\B_{\F_{3}}(\sd,\sd)$ is Morita equivalent to the path algebra
\[\F_3\left[{\xymatrix@C-0.5em@R-0.5em{
		\tilde \e_5 & \tilde \e_3\ar@/^2pc/[rr]^-{\hspace{-0.7mm}\ttau_2} & &\tilde \e_6\ar@/^2pc/[ll]^-{\hspace{-0.7mm}\ttau_1}\ar@/_2pc/[rr]_-{\hspace{-0.7mm}\ttau_3} & &  \tilde \e_4\ar@/_2pc/[ll]_-{\hspace{-0.7mm}\ttau_4}
}}\right] /(\ttau_4\ttau_3,~\ttau_4\ttau_1,~\ttau_2\ttau_1,~\ttau_2\ttau_3 )~,\] cf.\ \Cref{pfadsdF3}.

\section{Preliminaries on bisets and the double Burnside algebra}\label{prelim}

{\it Bisets.} Recall that an \textit{$(G,G)$-biset $X$} is a finite set $X$ together with a left $G$ and a right $G$-action that commute with each other, i.e.\ \mbox{$(h\cdot x)\cdot g = h\cdot (x\cdot g)=:h\cdot x\cdot g$} for $h,g\in G$ and $x\in X$.

Every $(G,G)$-biset $X$ can be regarded as a left $(G\ti G)$-set by setting $(h,g)x:=hxg^{-1}$ for $(h,g)\in G\ti G$ and $x\in X$. Likewise, every  left $(G\ti G)$-set $Y$ can be regarded as an $(G,G)$-biset by setting $h\cdot y\cdot g:=(h,g^{-1})y$ for $h,g\in G$ and $y\in Y$.
We freely use this identification.

{\it Tensor product.} Let  $M$ be an $(G,G)$-biset and let $N$ be a $(G,G)$-biset. The cartesian product $M\ti N$ is a $(G,G)$-biset via $h(m,n)p= (hm,np)$ for $h,p\in G$ and $(m,n)\in M\ti N$. It becomes a left $G$-set via $g(m,n)= (mg^{-1},gn)$ for  $g\in G$ and $(m,n)\in M\ti N$. We call the set of $G$-orbits on $M\ti N$ 
the {\it tensor product}  $M\tiG N$  of $M$ and $N$. This also is an $(G,G)$-biset. The $G$-orbit of the element $(m,n)\in M\ti N$ is denoted by $m\tiG n\in M\tiG N$.
Moreover, let $L$ be a $(G,G)$-biset. Then $M\tiG ( N\tiG L) \xrightarrow{\sim}{} (M\tiG N)\tiG L$, $m\tiG (n\tiG \ell) \mapsto  (m\tiG n)\tiG \ell$ as $(G,G)$-bisets.

{\it Double Burnside $R$-algebra.} We denote by $\B_R(G,G)$ the double Burnside $R$-algebra of $G$. Recall that $\B_R(G,G)$ is the $R$-module freely generated by the isomorphism classes of finite $(G,G)$-bisets, modulo the relations $[M\sqcup N] = [M]+[N]\text{ for $(G,G)$-bisets $M,N$.}$ Multiplication is defined by $[M]\cdotG[N]= [M\tiG N]\text{ for $(G,G)$-bisets $M,N$.}$ An $R$-linear basis of $\B_R(G,G)$ is given by $([(G\ti G)/U]: U\in\mathcal{L}_{G\ti G})$, where we choose a system $\mathcal{L}_{G\ti G}$ of representatives for the conjugacy classes of subgroups of $G\ti G$. Moreover, $1_{\B_\Z(G,G)}=[G]$.

{\it Abbreviation}. In case of $G=\sd$, we often abbreviate $\B_R:=\B_R(\sd,\sd)$. 

\section{$\Z$-linear basis of \texorpdfstring{$\B_\Z(\sd,\sd)$}{B\_Z(S\textthreeinferior,S\textthreeinferior)}}\label{Basis}
The following calculations were done using the computer algebra system Magma \cite{Magma}.

The group $\sd$ has the subgroups \begin{center}$V_0:=\{\idd\}$, $V_1:=\langle (1,2)\rangle$, $V_2:=\langle (1,3)\rangle$, $V_3:=\langle (2,3)\rangle$, $V_4:=\langle (1,2,3)\rangle$, $V_5:=\sd$ . \end{center}
The set $\{V_0$, $V_1$, $V_4$, $V_5\}$ is a system of representatives for the conjugacy classes of subgroups of $\sd$. In $\sd$, we write $a:=(1,2)$, $b:=(1,2,3)$ and $1:=\idd$. So $ V_1=\langle a\rangle$, $V_4=\langle b\rangle$ and $V_5=\langle a,b\rangle$.

A system of representatives for the conjugacy classes of subgroups of $\sd\ti \sd$ is given by
\[\enger
\begin{array}{lccclccc}
U_{0,0}&:=& V_0\ti V_0 &=& \{(1,1)\},\\
U_{1,0}&:=& V_1\ti V_0 &=& \langle(a,1)\rangle,\\
U_{0,1}&:=& V_0\ti V_1 &=& \langle(1,a)\rangle,\\
\Delta(V_1)&=& \phantom{V_0\ti V_0 } &\phantom{=}&\langle(a,a)\rangle,\\
U_{4,0}&:=& V_4\ti V_0 &=&\langle(b,1)\rangle,\\
U_{0,4}&:=& V_0\ti V_4 &=& \langle(1,b)\rangle,\\
\Delta(V_4)&=&\phantom{V_0\ti V_0 } &\phantom{=}&\langle(b,b)\rangle,\\
U_{1,1}&:=& V_1\ti V_1 &=& \langle(a,1), (1, a)\rangle,\\
U_{5,0}&:=& V_5\ti V_0 &=&\langle(a,1),(b,1)\rangle,\\
U_{0,5}&:=& V_0\ti V_5 &=& \langle(1,a),(1,b)\rangle,\\
U_6 &:=&\phantom{V_0\ti V_0 } &\phantom{=}& \langle(a,a),(1,b)\rangle,\end{array}\begin{array}{lccclccc}
U_{4,1}&:=& V_4\ti V_1 &=& \langle(b,1), (1, a)\rangle,\\
U_{1,4}&:=& V_1\ti V_4 &=& \langle(a,1), (1, b)\rangle,\\
U_7 &:=&\phantom{V_0\ti V_0 } &\phantom{=}& \langle(a,a),(b,1)\rangle,\\
\Delta(V_5)&=&\phantom{V_0\ti V_0 } &\phantom{=}&\langle(a,a),(b,b)\rangle,\\
U_{4,4}&:=& V_4\ti V_4 &=& \langle(b,1), (1, b)\rangle,\\
U_{1,5}&:=& V_1\ti V_5 &=& \langle(a,1), (1,a),(1, b)\rangle,\\
U_{5,1}&:=& V_5\ti V_1 &=& \langle(a,1),(b,1), (1, a)\rangle,\\
U_{4,5}&:=& V_4\ti V_5 &=& \langle(b,1), (1,a),(1, b)\rangle,\\
U_{5,4}&:=& V_5\ti V_4 &=& \langle(a,1),(b,1), (1, b)\rangle,\\
U_8&:=&\phantom{V_0\ti V_0 } &\phantom{=}& \langle(a,a),(b,1), (1, b)\rangle,\\
U_{5,5}&:=& V_5\ti V_5 &=& \langle(a,1), (1, a), (b,1), (1, b)\rangle .
\end{array}\]

Let $H_{i,j}:= [(\sd\ti\sd)/U_{i,j}]$ for $i,j\in \{0,1,4,5\}$, $H_s:=[(\sd\ti\sd)/U_{s}]$ for $s\in [6,8]$ and $H^\Delta_t:=[(\sd\ti\sd)/\Delta(V_t)]$ for $t\in \{1,4,5\}$.

So we obtain the $\Z$-linear basis \[\mathcal{H}:=\begin{array}{ll} 
&(H_{0,0},  	H_{1,0} ,       H_{0,1},      H^\Delta_1        ,    H_{4,0} ,       H_{0,4} ,    H^\Delta_4,		H_{1,1}      ,      H_{5,0} ,    H_{0,5} ,      H_6, \\
&H_{4,1}   , H_{1,4}         ,   H_7       , H^\Delta_5  ,    H_{4,4}   ,    H_{1,5}      ,       H_{5,1}  ,      H_{4,5}  ,   H_{5,4}    ,   H_8   ,    H_{5,5})\end{array}\] of $\B_\Z(\sd,\sd)$. Of course, $\mathcal{H}$ is also a $\Q$-linear basis of $\B_\Q(\sd,\sd)$.

\section{\texorpdfstring{$\B_\Q(\sd,\sd)$}{B\_Q(S\textthreeinferior,S\textthreeinferior)}}\label{rational} 
\subsection{\texorpdfstring{Peirce decomposition of $\B_\Q(\sd,\sd)$}{Peirce decomposition of B\_Q(S\textthreeinferior,S\textthreeinferior)}}\label{Ab23}
Using Magma \cite{Magma} we obtain an orthogonal decomposition of $1_{\B_\Q}$ into the following idempotents of $\B_\Q=\B_\Q(\sd,\sd)$.
\[\footnotesize\begin{array}{rclcrcl}
e&:=& -\frac{1}{2}H_{0,0}+H_{1,0}+\frac{1}{2}H_{4,0} &   {\eps}_2&:=&    -H_{0,0} + H_{1,0} + H_{0,1} + H^\Delta_1 - 2H_{1,1} \\
g&:=&\frac{4}{3}H_{0,0}-2H_{1,0}-\frac{4}{3}H_{0,1}- H_{4,0}+2H_{1,1}+ H_{4,1}& {\eps}_3&:=&    -\frac{1}{4}H_{0,0} + \frac{1}{4}H_{4,0} + \frac{1}{4}H_{0,4} + \frac{1}{2}H^\Delta_4 - \frac{3}{4}H_{4,4}\\
h&:=& -\frac{1}{12}H_{0,0}+\frac{1}{3}H_{0,1}+\frac{1}{4} H_{4,0}-\frac{1}{4} H_{0,4}+\frac{3}{4} H_{4,4}- H_{4,1} & {\eps}_4&:=&    \frac{1}{2}H_{0,0} - H^\Delta_1 - \frac{1}{2}H^\Delta_4 + H^\Delta_5
\end{array}\]Write $\eps_1:=e+g+h$. In \Cref{prim1} and \Cref{prim2}, we shall see that these idempotents are primitive.\\In a next step, we fix $\Q$-linear bases of the Peirce components. 
\[ \footnotesize  \hspace*{-0.6cm}\begin{array}{c|l}
	\begin{array}{l}{\text{Peirce}}\\\text{component}\hspace*{-4pt}\end{array} &{\begin{array}{l}{\Q\text{-linear basis}}\\[-0.1pt]{\phantom{\text{component}}}\end{array}} \Tstrut\Bstrut \\ \hline
	e\B_\Q e & e=-\frac{1}{2}H_{0,0}+H_{1,0}+\frac{1}{2}H_{4,0}  \Tstrut\Bstrut \\  \hline
	e\B_\Q g & b_{e,g}:=\frac{1}{2}H_{0,0} - H_{1,0} - \frac{1}{2}H_{0,1} - \frac{1}{2}H_{4,0} + H_{1,1} + \frac{1}{2}H_{4,1}
	\Tstrut\Bstrut \\  \hline
	e\B_\Q h & b_{e,h}:=-\frac{1}{8}H_{0,0} + \frac{1}{4}H_{1,0} + \frac{1}{2}H_{0,1} + \frac{1}{8}H_{4,0} - \frac{3}{8}H_{0,4} - H_{1,1} - \frac{1}{2}H_{4,1} + \frac{3}{4}H_{1,4} + \frac{3}{8}H_{4,4}
	\Tstrut\Bstrut \\  \hline
	
	g\B_\Q e & b_{g,e}:=-\frac{4}{3}H_{0,0} + 2H_{1,0} + H_{4,0} \Tstrut\Bstrut \\  \hline
	g\B_\Q g & g=\frac{4}{3}H_{0,0}-2H_{1,0}-\frac{4}{3}H_{0,1}- H_{4,0}+2H_{1,1}+ H_{4,1} \Tstrut\Bstrut \\  \hline
	g\B_\Q h & b_{g,h}:=-\frac{1}{3}H_{0,0} + \frac{1}{2}H_{1,0} + \frac{4}{3}H_{0,1} + \frac{1}{4}H_{4,0} - H_{0,4} - 2H_{1,1} - H_{4,1} + \frac{3}{2}H_{1,4} + \frac{3}{4}H_{4,4}\Tstrut\Bstrut \\  \hline
	
	h\B_\Q e & b_{h,e}:=-\frac{1}{3}H_{0,0} + H_{4,0}
	\Tstrut\Bstrut \\  \hline
	h\B_\Q g & b_{h,g}:=\frac{1}{3}H_{0,0} -\frac{1}{3}H_{0,1} - H_{4,0} + H_{4,1}
	\Tstrut\Bstrut \\  \hline
	h\B_\Q h & h=-\frac{1}{12}H_{0,0}+\frac{1}{3}H_{0,1}+\frac{1}{4} H_{4,0}-\frac{1}{4} H_{0,4}+\frac{3}{4} H_{4,4}- H_{4,1}
	\Tstrut\Bstrut \\  \hline
	e\B_\Q\eps_4 & b_{e,\eps_4}:=-\frac{1}{8}H_{0,0} + \frac{1}{4}H_{1,0} + \frac{1}{4}H_{0,1} + \frac{1}{8}H_{4,0}+ \frac{1}{8}H_{0,4} - \frac{1}{2}H_{1,1} -  \frac{1}{4}H_{0,5} - \frac{1}{4}H_{4,1} \\
	& \phantom{b_{h,\eps_4}:=}-\frac{1}{4}H_{1,4} - \frac{1}{8}H_{4,4} +  \frac{1}{2}H_{1,5} + \frac{1}{4}H_{4,5}\Tstrut\Bstrut \\  \hline 
	g\B_\Q\eps_4 & b_{g,\eps_4}:= -\frac{1}{3}H_{0,0} + \frac{1}{2}H_{1,0} + \frac{2}{3}H_{0,1} + \frac{1}{4}H_{4,0} 
	+ \frac{1}{3}H_{0,4} -H_{1,1}-  \frac{2}{3}H_{0,5} - \frac{1}{2}H_{4,1} \Tstrut\Bstrut \\ 
	& \phantom{b_{g,\eps_4}:=}-\frac{1}{2}H_{1,4} - \frac{1}{4}H_{4,4} +  H_{1,5} + \frac{1}{2}H_{4,5}\Tstrut\Bstrut \\  \hline 
	
	h\B_\Q\eps_4 & b_{h,\eps_4}:=-\frac{1}{12}H_{0,0} + \frac{1}{6}H_{0,1} + \frac{1}{4}H_{4,0} + \frac{1}{12}H_{0,4} 
	- \frac{1}{6}H_{0,5} -  \frac{1}{2}H_{4,1} - \frac{1}{4}H_{4,4} +\frac{1}{2}H_{4,5}\Tstrut\Bstrut \\  \hline 
	
	\eps_2\B_\Q\eps_2 & \eps_2= -H_{0,0} + H_{1,0} + H_{0,1} + H^\Delta_1 - 2H_{1,1} \Tstrut\Bstrut \\ \hline
	\eps_2\B_\Q\eps_4 & b_{\eps_2,\eps_4}:= -\frac{1}{2}H_{0,0} + \frac{1}{2}H_{1,0} + \frac{1}{2}H_{0,1} + \frac{1}{2}H^\Delta_1 
	+ \frac{1}{2}H_{0,4} -H_{1,1}-  \frac{1}{2}H_{0,5} - \frac{1}{2}H_6 -\frac{1}{2}H_{1,4}\\
	& \phantom{b_{\eps_2,\eps_4}:=}+H_{1,5}\Tstrut\Bstrut \\  \hline

	\eps_3\B_\Q\eps_3 & \eps_3=-\frac{1}{4}H_{0,0} + \frac{1}{4}H_{4,0} + \frac{1}{4}H_{0,4} + \frac{1}{2}H^\Delta_4 - \frac{3}{4}H_{4,4} \Tstrut\Bstrut \\ \hline
	
	\eps_4\B_\Q e & b_{\eps_4,e}:= \frac{1}{6}H_{0,0} - \frac{1}{3}H_{1,0} - \frac{1}{6}H_{4,0} + \frac{1}{3}H_{5,0} \Tstrut\Bstrut \\  \hline
	
	\eps_4\B_\Q g & b_{\eps_4,g}:= -\frac{1}{6}H_{0,0} + \frac{1}{3}H_{1,0} + \frac{1}{6}H_{0,1} + \frac{1}{6}H_{4,0} - \frac{1}{3}H_{1,1} -  \frac{1}{3}H_{5,0} - \frac{1}{6}H_{4,1} +\frac{1}{3}H_{5,1}\Tstrut\Bstrut \\  \hline

	\eps_4\B_\Q h & b_{\eps_4,h}:=\frac{1}{24}H_{0,0} - \frac{1}{12}H_{1,0} - \frac{1}{6}H_{0,1} - \frac{1}{24}H_{4,0} + \frac{1}{8}H_{0,4} +  \frac{1}{3}H_{1,1} + \frac{1}{12}H_{5,0} \\
	& \phantom{b_{h,\eps_4}:=}+\frac{1}{6}H_{4,1} - \frac{1}{4}H_{1,4} -  \frac{1}{8}H_{4,4} - \frac{1}{3}H_{5,1} +\frac{1}{4}H_{5,4}\Tstrut\Bstrut \\  \hline 
	
	\eps_4\B_\Q\eps_2 & b_{\eps_4,\eps_2}:= -\frac{1}{2}H_{0,0} 
	+\frac{1}{2}H_{1,0} + \frac{1}{2}H_{0,1} + \frac{1}{2}H^\Delta_1 + \frac{1}{2}H_{4,0} -H_{1,1} - \frac{1}{2}H_{5,0} -\frac{1}{2}H_{4,1}\\ 
	& \phantom{b_{\eps_2,\eps_4}:=}-\frac{1}{2}H_7 + H_{5,1}  \Tstrut\Bstrut \\ \hline
	\eps_4\B_\Q\eps_4 & \eps_4= \frac{1}{2}H_{0,0} - H^\Delta_1 - \frac{1}{2}H^\Delta_4 + H^\Delta_5,\\
	& b'_{\eps_4,\eps_4}:=\frac{1}{24}H_{0,0} - \frac{1}{12}H_{1,0} - \frac{1}{12}H_{0,1} - \frac{1}{24}H_{4,0} 
	- \frac{1}{24}H_{0,4} +\frac{1}{6}H_{1,1}+  \frac{1}{12}H_{5,0} + \frac{1}{12}H_{0,5}  \\
	& \phantom{b_{g,\eps_4}:=}+\frac{1}{12}H_{4,1}+\frac{1}{12}H_{1,4} + \frac{1}{24}H_{4,4} - \frac{1}{6} H_{1,5} - \frac{1}{6}H_{5,1} - \frac{1}{12}H_{4,5} - \frac{1}{12}H_{5,4} + \frac{1}{6}H_{5,5}~,\\
	&b''_{\eps_4,\eps_4}:=\frac{1}{4}H_{0,0} - \frac{3}{4}H_{1,0} - \frac{3}{4}H_{0,1} + \frac{1}{4}H^\Delta_1 
	- \frac{1}{4}H_{4,0} -\frac{1}{4}H_{0,4}+  \frac{3}{2}H_{1,1} + \frac{3}{4}H_{5,0}  \\
	& \phantom{b_{g,\eps_4}:=}+\frac{3}{4}H_{0,5}-\frac{1}{4}H_6  + \frac{3}{4}H_{4,1} + \frac{3}{4} H_{1,4} - \frac{1}{4}H_7 + \frac{1}{4}H_{4,4} - \frac{3}{2}H_{1,5}- \frac{3}{2}H_{5,1}\\
	& \phantom{b_{g,\eps_4}:=}-\frac{3}{4}H_{4,5}-\frac{3}{4}H_{5,4}+\frac{1}{4}H_8 +\frac{3}{2}H_{5,5}
	\\
	\end{array}\] \begin{Remark}\label{prim1}\rm The idempotents $e,g,h,\eps_2,\eps_3$ are primitive, as $e\B_\Q e\cong \Q$, $g\B_\Q g\cong \Q$, $h\B_\Q h\cong \Q$, $\eps_2\B_\Q \eps_2\cong \Q$ and $\eps_3\B_\Q \eps_3\cong \Q$. 
\end{Remark}We have the following multiplication table for the basis elements of $\B_\Q=\B_\Q(\sd,\sd)$.\[\hspace*{-1cm}\footnotesize {\setlength{\arraycolsep}{0.1pt}
\begin{array}{c|c|c|c|c|c|c|c|c|c|c|c|c|c|c|c|c|c|c|c|c|c|c}
(\cdot) & e & b_{e,g} & b_{e,h} & b_{g,e} & g & b_{g,h} & b_{h,e} & b_{h,g} & h & b_{e,\eps_4} & b_{g,\eps_4} & b_{h,\eps_4} &\eps_2 & b_{\eps_2,\eps_4}  & \eps_3 & b_{\eps_4,e} & b_{\eps_4,g} & b_{\eps_4,h} & b_{\eps_4,\eps_2} &  \eps_4 &b'_{\eps_4,\eps_4} & b''_{\eps_4,\eps_4}\\ \hline    

e & e  &b_{e,g} & b_{e,h} & 0 & 0 & 0 &0 &0& 0& b_{e,\eps_4} &0 &0 & 0 &0 &0 & 0 &0 &0 &0 &0 &0 &0 \\ \hline

b_{e,g} & 0 &0  &0 & e & b_{e,g} & b_{e,h}  &0 &0& 0&0 & b_{e,\eps_4} &0 &0&0 &0 & 0 &0 &0 &0 &0 &0 &0 \\ \hline

b_{e,h} & 0  &0 & 0 & 0 & 0 & 0 &e &b_{e,g}& b_{e,h}&0 & 0 &b_{e,\eps_4} & 0 &0 &0 & 0 &0 &0 &0 &0 &0 &0 \\ \hline

b_{g,e} & b_{g,e}  &g &b_{g,h}  & 0 & 0 & 0 &0 &0& 0&b_{g,\eps_4} & 0 &0 & 0 &0 &0 & 0 &0 &0 &0 &0 &0 &0 \\ \hline

g & 0  &0 &0  & b_{g,e} & g & b_{g,h} &0 &0& 0&0&b_{g,\eps_4}  &0 & 0 &0 &0 & 0 &0 &0 &0 &0 &0 &0 \\ \hline

b_{g,h} & 0  &0 &0  & 0 & 0 & 0 &b_{g,e} &g& b_{g,h}&0&0   & b_{g,\eps_4}&0 &0 &0 & 0 &0 &0 &0 &0 &0 &0 \\ \hline

b_{h,e} & b_{h,e}  &b_{h,g} & h  & 0 & 0 & 0 &0 &0& 0&b_{h,\eps_4}&0   & 0&0 &0 &0 & 0 &0 &0 &0 &0 &0 &0 \\ \hline

b_{h,g} & 0  &0 & 0  & b_{h,e} & b_{h,g} & h &0 &0&0& 0&b_{h,\eps_4}   & 0&0 &0 &0 & 0 &0 &0 &0 &0 &0 &0 \\ \hline

h & 0  &0 & 0  & 0 & 0 & 0 &b_{h,e} &b_{h,g}&h& 0&0&b_{h,\eps_4}   & 0 &0 &0 & 0 &0 &0 &0 &0 &0 &0 \\ \hline

b_{e,\eps_4} & 0  &0 & 0  & 0 & 0 & 0 &0 &0&0& 0&0   & 0&0 &0 &0 & 0 &0 &0 &0 &b_{e,\eps_4} &0 &0 \\ \hline

b_{g,\eps_4} & 0  &0 & 0  & 0 & 0 & 0 &0 &0&0& 0&0   & 0&0 &0 &0 & 0 &0 &0 &0 &b_{g,\eps_4} &0 &0 \\ \hline

b_{h,\eps_4} & 0  &0 & 0  & 0 & 0 & 0 &0 &0&0& 0&0   & 0&0 &0 &0 & 0 &0 &0 &0 &b_{h,\eps_4} &0 &0 \\ \hline

\eps_2 & 0  &0 & 0  & 0 & 0 & 0 &0 &0&0& 0&0   & 0&\eps_2 &b_{\eps_2,\eps_4} &0 & 0 &0 &0 &0 &0 &0 &0 \\ \hline

b_{\eps_2,\eps_4} & 0  &0 & 0  & 0 & 0 & 0 &0 &0&0& 0&0   & 0&0 &0 &0 & 0 &0 &0 &0 &b_{\eps_2,\eps_4} &0 &0 \\ \hline

\eps_3 & 0  &0 & 0  & 0 & 0 & 0 &0 &0&0& 0&0   & 0&0 &0 &\eps_3 & 0 &0 &0 &0 &0 &0 &0 \\ \hline

b_{\eps_4,e} & b_{\eps_4,e}  &b_{\eps_4,g} & b_{\eps_4,h}  & 0 & 0 & 0 &0 &0&0& b'_{\eps_4,\eps_4}&0   & 0&0 &0 &0 & 0 &0 &0 &0 &0 &0 &0 \\ \hline

b_{\eps_4,g} & 0  &0 & 0  & b_{\eps_4,e} & b_{\eps_4,g} & b_{\eps_4,h} &0 &0&0& 0&b'_{\eps_4,\eps_4}   & 0&0 &0 &0 & 0 &0 &0 &0 &0 &0 &0 \\ \hline

b_{\eps_4,h} & 0  &0 &0  & 0 & 0 & 0 &b_{\eps_4,e} &b_{\eps_4,g}& b_{\eps_4,h}&0&0   & b'_{\eps_4,\eps_4}&0 &0 &0 & 0 &0 &0 &0 &0 &0 &0 \\ \hline

b_{\eps_4,\eps_2} & 0  &0 & 0  & 0 & 0 & 0 &0 &0&0& 0&0   & 0&b_{\eps_4,\eps_2} &b''_{\eps_4,\eps_4}-12b'_{\eps_4,\eps_4} &0 & 0 &0 &0 &0 &0 &0 &0 \\ \hline

\eps_4 & 0  &0 & 0  & 0 & 0 & 0 &0 &0&0& 0&0   & 0&0 &0 &0 & b_{\eps_4,e} &b_{\eps_4,g} &b_{\eps_4,h} &b_{\eps_4,\eps_2} &\eps_4 &b'_{\eps_4,\eps_4} &b''_{\eps_4,\eps_4} \\ \hline

b'_{\eps_4,\eps_4} & 0  &0 & 0  & 0 & 0 & 0 &0 &0&0& 0&0   & 0&0 &0 &0 & 0 &0 &0 &0 &b'_{\eps_4,\eps_4} &0 &0 \\ \hline

b''_{\eps_4,\eps_4} & 0  &0 & 0  & 0 & 0 & 0 &0 &0&0& 0&0   & 0&0 &0 &0 & 0 &0 &0 &0 &b''_{\eps_4,\eps_4} &0 &0 \\ 
\end{array}}
\]
We see that $\eps_3$ is even a central element.

\begin{Lemma}\label{ga4}
	Consider $\Q[\eta,\xi]/(\eta^2,\eta\xi,\xi^2) = \Q[\overline{\eta},\overline{\xi}]$, where we let $\overline{\xi} := \xi + (\eta^2,\eta\xi,\xi^2)$ and 
	\mbox{$\overline{\eta} := \eta + (\eta^2,\eta\xi,\xi^2)$}.
	
	We have the $\Q$-algebra isomorphism
	\[\begin{array}{cccc}
	\mu:&\Q[\overline{\eta},\overline{\xi}]&\rightarrow& \eps_4\B_\Q\eps_4 \\
	&\overline{\eta}&\mapsto& b'_{\eps_4,\eps_4}\\
	&\overline{\xi}&\mapsto& b''_{\eps_4,\eps_4}~.
	\end{array}\]
\end{Lemma}

\begin{proof}
	Since $\eps_4\B_\Q\eps_4 = {_\Q\langle} \eps_4, b'_{\eps_4,\eps_4},  b''_{\eps_4,\eps_4}\rangle$ is commutative and $(b'_{\eps_4,\eps_4})^2=0$, $(b''_{\eps_4,\eps_4})^2=0$ and $b'_{\eps_4,\eps_4}b''_{\eps_4,\eps_4}=0~$, the map $\mu$ is a well-defined $\Q$-algebra morphism. 
	
	As the $\Q$-linear basis $(1,\overline{\eta},\overline{\xi})$ is mapped to the $\Q$-linear basis $(\eps_4, b'_{\eps_4,\eps_4},b''_{\eps_4,\eps_4})$, it is bijective.
\end{proof}

\begin{Remark}\label{prim2}
	The ring $\Q[\overline{\eta},\overline{\xi}]$ is local. In particular, $\eps_4$ is a primitive idempotent of $\B_\Q$\hspace*{0,2mm}.
\end{Remark}

\begin{proof}
	We have $\UU(\Q[\overline{\eta},\overline{\xi}])= \Q[\overline{\eta},\overline{\xi}]\setminus (\overline{\eta},\overline{\xi})$, as for $u:= a + b\overline{\eta} +c\overline{\xi}$ the inverse is given by $u^{-1}= a^{-1} - a^{-2}b\overline{\eta} -a^{-2}c\overline{\xi}$ for $a,b,c\in\Q$, with $a\neq 0$.
	Thus the nonunits of $\Q[\overline{\eta},\overline{\xi}]$ form an ideal and so  $\Q[\overline{\eta},\overline{\xi}]$ is a local ring.  
\end{proof}

To standardize notation, we aim to construct a  $\Q$-algebra $A:=\bigoplus\limits_{i,j} A_{i,j}$ with
\mbox{$A\cong \B_\Q(\sd,\sd)$.} 

In a first step to do so, we choose $\Q$-vector spaces $A_{i,j}$ and $\Q$-linear isomorphisms \mbox{$\ga_{i,j}:A_{i,j}\xrightarrow{\sim} \eps_i \B_\Q\eps_j$} for $i,j\in [1,4]$.
We define the tuple of $\Q$-vector spaces \[\begin{array}{lllllll}
(A_{1,1}~, & & A_{1,2}~, & & A_{1,3}~, & & A_{1,4}~, \\
\phantom{(}A_{2,1}~, & & A_{2,2}~, & & A_{2,3}~, & & A_{2,4}~, \\
\phantom{(} A_{3,1}~, & & A_{3,2}~, & & A_{3,3}~, & & A_{3,4}~, \\
\phantom{(} A_{4,1}~, & & A_{4,2}~, & & A_{4,3}~, & & A_{4,4}) \\
\end{array} ~:=~
\begin{array}{lllllll}
(\Q^{3\ti 3}~, & & 0~, & & 0~, & & \Q^{3\ti 1}~, \\
\phantom{(}   0~, & & \Q~, & & 0~, & & \Q~, \\
\phantom{(}   0~, & & 0~, & & \Q~, & & 0~, \\
\phantom{(}  \Q^{1\ti 3}~, & & \Q~, & & 0~, & & \Q[\overline{\eta},\overline{\xi}])~, \rm cf.\ \Cref{ga4}.
\end{array}\]

We have $\gamma_{s,t}=0$ for $(s,t)\in\{(1,2),(1,3), (2,1), (2,3), (3,1), (3,2), (3,4), (4,3)\}$.

Let \[{\weiter\begin{array}{lll}
	&\begin{array}{rccl}
	\ga_{1,1}:~A_{1,1}&\xrightarrow{\sim}&\eps_1\B_\Q\eps_1 \\
	\begin{pmatrix}
	r_{1,1} & r_{1,2} & r_{1,3}\\
	r_{2,1} & r_{2,2} & r_{2,3}\\
	r_{3,1} & r_{3,2} & r_{3,3}\\
	\end{pmatrix}      &\mapsto & \enger\begin{array}{llllll}
	&r_{1,1}e &+& r_{1,2}b_{e,g} &+& r_{1,3}b_{e,h}\\
	+& r_{2,1}b_{g,e} &+& r_{2,2} g & +&r_{2,3}b_{g,h}\\
	+&r_{3,1}b_{h,e} &+& r_{3,2} b_{h,g} &+& r_{3,3}h
	\end{array}\end{array}, & \begin{array}{rccl}
	\ga_{1,4}\phantom{:=\mu}:~A_{1,4}&\xrightarrow{\sim}&\eps_1\B_\Q\eps_4 \\
	\begin{pmatrix}
	u_1\\u_2\\u_3
	\end{pmatrix} &\mapsto &\begin{array}{rrr} u_1b_{e,\eps_4}\\+\hspace*{1mm}u_2b_{g,\eps_4} \\+\hspace*{1mm}u_3b_{h,\eps_4}\end{array}\end{array},\\\\
	
	&\begin{array}{rccl}\ga_{2,2}\hspace*{12mm}:~A_{2,2}&\xrightarrow{\sim}&\eps_2\B_\Q\eps_2\\
	u &\mapsto&u\eps_2\end{array}, &\hspace*{-2.8cm} \begin{array}{rccl}
	\ga_{2,4}\phantom{:=\mu}:~A_{2,4}&\xrightarrow{\sim}&\eps_2\B_\Q\eps_4\\
	u &\mapsto&ub_{\eps_2,\eps_4}\end{array},\\\\
	&\begin{array}{rccl}
	\ga_{3,3}\phantom{:=\nu^{-1}}\hspace*{2.8mm}:~A_{3,3}&\xrightarrow{\sim}&\eps_3\B_\Q\eps_3\\
	u &\mapsto&u\eps_3\end{array}, &\hspace*{-2.8cm} \begin{array}{rccl}
	\ga_{4,1}\phantom{:=\mu}:~A_{4,1}&\xrightarrow{\sim}&\eps_4\B_\Q\eps_1 \\
	\begin{pmatrix}
	v_1&v_2&v_3
	\end{pmatrix} &\mapsto &  v_1b_{\eps_4,e} + v_2b_{\eps_4,g} + v_3b_{\eps_4,h} \end{array},\\\\
	&\begin{array}{rccl}
	\ga_{4,2}\phantom{:=\nu^{-1}}\hspace*{2.8mm}:~A_{4,2}&\xrightarrow{\sim}&\eps_4\B_\Q\eps_2\\
	u &\mapsto&u b_{\eps_4,\eps_2}\end{array}, &\hspace*{-2.8cm} \begin{array}{rccl}
	\ga_{4,4}\stackrel{\rm L.\ref{ga4}}{:=}\mu:~A_{4,4}&\xrightarrow{\sim}&\eps_4\B_\Q\eps_4\\
	a+b\overline{\eta}+c\overline{\xi} &\mapsto&a\eps_4+bb'_{\eps_4,\eps_4}+cb''_{\eps_4,\eps_4}.
	\end{array}\end{array}}\]

Let $\beta: \B_\Q\ti \B_\Q\rightarrow \B_\Q$ be the multiplication map on $\B_\Q$. Write \[\be_{i,j,k}:=\beta |^{\eps_i\B_\Q\eps_k}_{\eps_i\B_\Q\eps_j \ti \eps_j\B_\Q\eps_k} : \eps_i\B_\Q\eps_j \ti \eps_j\B_\Q\eps_k\rightarrow \eps_i\B_\Q\eps_k~.\]
Now, we construct $\Q$-bilinear multiplication maps $\al_{i,j,k}$ for $i,j,k\in[1,4]$ such that the following quadrangle of maps commutes.

\[
\begin{array}{l}
\xymatrix@C=50mm{
	A_{i,j}\ti A_{j,k}\ar[r]^{\al_{i,j,k}}\ar[d]_{\ga_{i,j}\ti\ga_{j,k}} & A_{i,k}\ar[d]_{\ga_{i,k}}  \\
	\eps_i\B_\Q\eps_j\ti\eps_j\B_\Q\eps_k\ar[r]^{\beta_{i,j,k}} &\eps_i\B_\Q\eps_k\\
}\\
\end{array}
\] I.e.\ we set $\al_{i,j,k}:=\ga_{i,k}^{-1}\circ\be_{i,j,k}\circ(\ga_{i,j}\ti\ga_{j,k})$. This leads to \begin{itemize}
	\item  $\al_{i,j,k}=0$\\ if $(i,j)$, $(j,k)$ or $(i,k)$ is contained in $\{(1,2),(1,3), (2,1), (2,3), (3,1), (3,2), (3,4), (4,3)\}$
	\item $\al_{1,1,1}: A_{1,1}\ti A_{1,1}\rightarrow A_{1,1}$, $(X,Y) \mapsto XY$
	
	\item $\al_{1,1,4}: A_{1,1}\ti A_{1,4}\rightarrow A_{1,4}$, $(X,u) \mapsto Xu $
	\item $\al_{1,4,1}=0$
	\item
	$\al_{1,4,4}: A_{1,4}\ti A_{4,4}\rightarrow A_{1,4}$, $(u,a+b\overline{\eta}+c\overline{\xi})\mapsto ua$
	\item
	$\al_{2,2,2}: A_{2,2}\ti A_{2,2}\rightarrow A_{2,2}$, $(u,v)\mapsto  uv$
	\item
	$\al_{2,2,4}: A_{2,2}\ti A_{2,4}\rightarrow A_{2,4}$, $(u,v) \mapsto uv$
	\item $\al_{2,4,2}=0$
	\item
	$\al_{2,4,4}: A_{2,4}\ti A_{4,4}\rightarrow A_{2,4}$, $(u,a+b\overline{\eta}+c\overline{\xi}) \mapsto ua$
	\item $\al_{3,3,3}: A_{3,3}\ti A_{3,3}\rightarrow A_{3,3}$, $	(u,v)\mapsto  uv$
	\item $\al_{4,1,1}: A_{4,1}\ti A_{1,1}\rightarrow A_{4,1}$, $(v,X)\mapsto 
	vX$
	\item
	$\al_{4,1,4}: A_{4,1}\ti A_{1,4}\rightarrow A_{4,4}$, $
	(v,u)\mapsto  
	vu\overline{\eta}$
	\item 
	$\al_{4,2,2}: A_{4,2}\ti A_{2,2}\rightarrow A_{4,2}$, $(u,v)\mapsto  uv$ 
	\item 
	$\al_{4,2,4}: A_{4,2}\ti A_{2,4}\rightarrow A_{4,4}$, $(u,v)\mapsto  uv(\overline{\xi}-12\overline{\eta})$
	\item 
	$\al_{4,4,1}: A_{4,4}\ti A_{4,1}\rightarrow A_{4,1}$, $
	(a+b\overline{\eta}+c\overline{\xi},v)\mapsto  av$
	\item $\al_{4,4,2}: A_{4,4}\ti A_{4,2}\rightarrow A_{4,2}$, $
	(a+b\overline{\eta}+c\overline{\xi}, v)\mapsto  av$
	\item $\al_{4,4,4}: A_{4,4}\ti A_{4,4}\rightarrow A_{4,4}$, $
	(a+b\overline{\eta}+c\overline{\xi},\tilde a+\tilde b\overline{\eta}+\tilde c\overline{\xi} )\mapsto   (a+b\overline{\eta}+c\overline{\xi})\cdot(\tilde a+\tilde b\overline{\eta}+\tilde c\overline{\xi} )$ where $a,b,c,\tilde a,\tilde b,\tilde c\in\Q$
\end{itemize} For convenience, we fix a notation similar to matrices and matrix multiplication.
\begin{Notation} Suppose given $r\in\Z_{\geq 0}\ab$. Suppose given $R$-modules $M_{i,j}$ for $i,j\in [1,r]$. 
	We write \[
	\bigoplus\limits_{i,j\in [1,r]}M_{i,j}=: \left[\begin{array}{ccccccc}
	M_{1,1} & & M_{1,2} && \dots && M_{1,r} \\
	M_{2,1} & & M_{2,2} && \dots && M_{2,r} \\
	\vdots  & & \vdots  & & \dots && \vdots\\
	M_{r,1} & & M_{r,2} && \dots && M_{r,r} \\ \end{array}\right]~.\]
	
	Accordingly, elements of this direct sum are written as matrices with entries in the respective summands, i.e.\ in the form $[m_{i,j}]_{i,j}$ with $m_{i,j}\in M_{i,j}$ for $i,j\in [1,r]$. \end{Notation}
\begin{Proposition}\label{gamma}Let \[A:=\bigoplus\limits_{i,j\in[1,4]}A_{i,j} =\left[\enger\begin{array}{lllllll}
	\phantom{(}A_{1,1} & & A_{1,2} & & A_{1,3} & & A_{1,4} \\
	\phantom{(}A_{2,1} & & A_{2,2} & & A_{2,3} & & A_{2,4} \\
	\phantom{(} A_{3,1} & & A_{3,2} & & A_{3,3} & & A_{3,4} \\
	\phantom{(} A_{4,1} & & A_{4,2} & & A_{4,3} & & A_{4,4} \\
	\end{array}\right]=
	\left[\enger\begin{array}{lcccccl}
	\Q^{3\ti 3} & & 0 & & 0 & & \Q^{3\ti 1} \\
	~0 & & \Q & & 0 & & \Q \\
	~0 & & 0 & & \Q & & ~0 \\
	\Q^{1\ti 3} & & \Q & & 0 & & \Q[\overline{\eta},\overline{\xi}] 
	\end{array}\right]~.\] Define the multiplication \[\begin{array}{rccccl} 
	A&\times & A &\rightarrow &A\\
	([a_{i,j}]_{i,j}&,&[a'_{s,t}]_{s,t}) &\mapsto & [\sum\limits_{r\in[1,4]} \al_{i,r,j}(a_{i,r},a'_{r,j})]_{i,j}
	\end{array}~.\] 
	We obtain a $\Q$-algebra isomorphism 
	\begin{center}
		$\begin{array}{rcl} A&\rxisoa{\gamma} &\B_\Q(\sd,\sd)\\\\
		
		[a_{i,j}]_{i,j\in[1,4]}&\mapsto&\sum\limits_{i,j\in[1,4]}\ga_{i,j}(a_{i,j})~.\end{array}$ 
	\end{center} 
	\end{Proposition}
	
 \subsection{\texorpdfstring{$\B_\Q(\sd,\sd)$}{B\_Q(S\textthreeinferior,S\textthreeinferior)} as path algebra modulo relations}
We aim to write $\B_\Q=\B_\Q(\sd,\sd)\cong A$, up to Morita equivalence, as a path algebra modulo relations.\\ 
We denote by  $\e_{i,j}\in A_{1,1}=\Q^{3\ti 3}$ the elements that have a single non-zero entry $1$ at position $(i,j)$.
We have $a_{1,1}:=\ga^{-1}(e)=\e_{1,1}\in\Q^{3\ti 3}\subseteq A$, $\ga^{-1}(g)=\e_{2,2}\in\Q^{3\ti 3}\subseteq A$, $\ga^{-1}(e)=\e_{3,3}\in\Q^{3\ti 3}\subseteq A$ and $a_{k,k}:=\ga^{-1}(\eps_k)$ for $k\in [2,4]$, cf.\ \Cref{gamma}. 

We have $Aa_{1,1}\cong A \e_{2,2}$ as $A$-modules, using multiplication with $\e_{1,2}$ from the right from $Aa_{1,1}$ to $A\e_{2,2}$ and multiplication with $\e_{2,1}$ from the right from $A\e_{2,2}$ to $Aa_{1,1}$. Note that $\e_{1,2}\e_{2,1}=a_{1,1}$ and $\e_{2,1}\e_{1,2}=\e_{2,2}$. Similarly $Aa_{1,1}\cong A\e_{3,3}$. 

Therefore, $A$ is Morita equivalent to 
\[
A':= {\scriptsize(\sum\limits_{i\in [1,4]}a_{i,i})A(\sum\limits_{i\in [1,4]}a_{i,i})=\bigoplus\limits_{i,j\in[1,4]}a_{i,i}Aa_{j,j}=\bigoplus\limits_{i,j\in[1,4]}a_{i,i}A_{i,j}a_{j,j}} \; .
\]

Write $A'_{i,j}:=a_{i,i}A_{i,j}a_{j,j}=A_{i,j}$ for $i,j\in[2,4]$.\\ Identify $A'_{1,1}:=\Q=\scriptsize\left(\begin{array}{ccc}\Q &0 & 0\\ 0&0&0\\0&0&0\end{array}\right)=a_{1,1}A_{1,1}a_{1,1}\subseteq A_{1,1}=\Q^{3\ti 3}$.\\ Identify $A'_{1,4}:=\Q=\scriptsize\left(\begin{array}{ccc}\Q\\ 0\\0\end{array}\right)=a_{1,1}A_{1,4}a_{4,4}\subseteq A_{1,4}=\Q^{3\ti 1}$.\\ Identify $A'_{4,1}:=\Q=\scriptsize\left(\begin{array}{ccc}\Q &0 & 0\end{array}\right)=a_{4,4}A_{4,1}a_{1,1}\subseteq A_{4,1}=\Q^{1\ti 3}$.  Let $A'_{1,j}:=0$ and $A'_{j,1}:=0$ for $j\in[2,3]$. 

We have the $\Q$-linear basis of $A'$
\[\footnotesize\begin{array}{rclcrcl}
a_{1,1}\phantom{:}= \left[\begin{array}{lllllll}
1 & & 0 & & 0 & & 0 \\
0 & & 0 & & 0 & & 0 \\
0 & & 0 & & 0 & & 0 \\
0 & & 0 & & 0 & & 0
\end{array}\right] &
a_{2,2}\phantom{:}=\left[\begin{array}{lllllll}
0 & & 0 & & 0 & & 0 \\
0 & & 1 & & 0 & & 0 \\
0 & & 0 & & 0 & & 0 \\
0 & & 0 & & 0 & & 0
\end{array} \right] &
a_{3,3\phantom{:}}= \left[\begin{array}{lllllll}
0 & & 0 & & 0 & & 0 \\
0 & & 0 & & 0 & & 0 \\
0 & & 0 & & 1 & & 0 \\
0 & & 0 & & 0 & & 0
\end{array}\right]  \\
a_{4,4}\phantom{:}= \left[\begin{array}{lllllll}
0 & & 0 & & 0 & & 0 \\
0 & & 0 & & 0 & & 0 \\
0 & & 0 & & 0 & & 0 \\
0 & & 0 & & 0 & & 1
\end{array}\right] &
a_{1,4} :=\left[\begin{array}{lllllll}
0 & & 0 & & 0 & & 1 \\
0 & & 0 & & 0 & & 0 \\
0 & & 0 & & 0 & & 0 \\
0 & & 0 & & 0 & & 0
\end{array}\right] &
a_{4,1}:= \left[\begin{array}{lllllll}
0 & & 0 & & 0 & & 0 \\
0 & & 0 & & 0 & & 0 \\
0 & & 0 & & 0 & & 0 \\
1 & & 0 & & 0 & & 0
\end{array}\right] \\
a_{2,4}:= \left[\begin{array}{lllllll}
0 & & 0 & & 0 & & 0 \\
0 & & 0 & & 0 & & 1 \\
0 & & 0 & & 0 & & 0 \\
0 & & 0 & & 0 & & 0
\end{array}\right]  &
a_{4,2}:=\left[\begin{array}{lllllll}
0 & & 0 & & 0 & & 0 \\
0 & & 0 & & 0 & & 0 \\
0 & & 0 & & 0 & & 0 \\
0 & & 1 & & 0 & & 0
\end{array}\right]  &
a'_{4,4}:=  \left[\begin{array}{lllllll}
0 & & 0 & & 0 & & 0 \\
0 & & 0 & & 0 & & 0 \\
0 & & 0 & & 0 & & 0 \\
0 & & 0 & & 0 & & \oeta
\end{array}\right]  \\
a''_{4,4}:= \left[\begin{array}{lllllll}
0 & & 0 & & 0 & & 0 \\
0 & & 0 & & 0 & & 0 \\
0 & & 0 & & 0 & & 0 \\
0 & & 0 & & 0 & & \oxi
\end{array}\right]  \\
\end{array}\]
We have the following multiplication table for the basis elements. 
\hspace*{-2cm}\[\footnotesize {\setlength{\arraycolsep}{2.9pt}
	\begin{array}{c|c|c|c|c|c|c|c|c|c|c}
	(\cdot) & a_{1,1} & a_{1,4} &a_{2,2} & a_{2,4}  & a_{3,3} & a_{4,1} &  a_{4,2} &  a_{4,4} &a'_{4,4} & a''_{4,4}\\ \hline    
	
	a_{1,1} & a_{1,1} & a_{1,4} &0 &0 & 0 &0 &0 & 0 &0 &0  \\ \hline
	
	a_{1,4}&0 &0 &0 & 0 &0 &0 &0 &a_{1,4} &0 &0 \\ \hline
	
	a_{2,2} &0   & 0&a_{2,2} &a_{2,4} &0 & 0 &0 &0 &0 &0  \\ \hline
	
	a_{2,4} &0 &0 &0 & 0 &0 &0 &0 &a_{2,4} &0 &0 \\ \hline
	
	a_{3,3} &0   & 0&0 &0 &a_{3,3} & 0 &0 &0 &0 &0  \\ \hline
	
	a_{4,1} & a_{4,1}  & a'_{4,4}&0   & 0&0 &0 &0 & 0 &0 &0   \\ \hline
	
	a_{4,2} &0   & 0&a_{4,2} &a''_{4,4}-12a'_{4,4} &0 & 0 &0 &0 &0 &0  \\ \hline
	
	a_{4,4} &0   & 0&0 &0 &0 & a_{4,1}&a_{4,2} &a_{4,4} &a'_{4,4} &a''_{4,4} \\ \hline
	
	a'_{4,4} &0 &0 &0 & 0 &0 &0 &0 &a'_{4,4} &0 &0 \\ \hline
	
	a''_{4,4} &0 &0 &0 & 0 &0 &0 &0 &a''_{4,4} &0 &0 \\
	\end{array}}
\]
We have $a'_{4,4} = a_{4,1}\cdot a_{1,4}$ and $a''_{4,4} =a_{4,2}\cdot a_{2,4} + 12a_{4,1}\cdot a_{1,4}$ . Hence, as a $\Q$-algebra $A'$ is generated by $a_{1,1},a_{2,2},a_{3,3},a_{4,4}, a_{1,4}, a_{4,1}, a_{2,4}, a_{4,2}$ .

Consider the quiver
$\Psi:=\left[{\xymatrix@C-0.5em@R-0.5em{
		\tilde a_{3,3} & \tilde a_{2,2}\ar@/^1pc/[rr]^-{\hspace{-0.7mm}\sigma} & &\tilde a_{4,4}\ar@/^1pc/[ll]^-{\hspace{-0.7mm}\vartheta}\ar@/_1pc/[rr]_-{\hspace{-0.7mm}\rho} & &  \tilde a_{1,1}\ar@/_1pc/[ll]_-{\hspace{-0.7mm}\pi}
}}\right]~.$

We have a surjective $\Q$-algebra morphism $\phi: \Q\Psi\rightarrow A'$ by sending 
\[\begin{array}{rccccccccccccccccl}
\tilde a_{1,1} &\mapsto& a_{1,1} &,&
\tilde a_{2,2} &\mapsto& a_{2,2} &, & 
\tilde a_{3,3} &\mapsto& a_{3,3}  &, &
\tilde a_{4,4} &\mapsto& a_{4,4}  &,&\\
\rho &\mapsto& a_{4,1}&,&
\pi &\mapsto& a_{1,4}&,&
\vartheta &\mapsto& a_{4,2}&,&
\sigma &\mapsto& a_{2,4}&.&
\end{array}\]

We establish the following multiplication trees, where we underline the elements that are not in a $\Q$-linear relation with previously underlined elements.

\[{\xymatrix@C-0.5em@R-0.5em{
		\ul{a_{1,1}}\ar[rr]_-{\hspace{-0.5mm}a_{1,4}}& & {\ul{a_{1,4}}}\ar[d]_-{\hspace{-0.5mm}a_{4,2}}\ar[rr]_-{\hspace{-0.5mm}a_{4,1}} && a_{1,4}a_{4,1}=0\\ 
		&  &a_{1,4}a_{4,2}=0 & & & \\
}}{\xymatrix@C-0.5em@R-0.5em{
		\ul{a_{2,2}}\ar[rr]_-{\hspace{-0.5mm}a_{2,4}}& & {\ul{a_{2,4}}}\ar[d]_-{\hspace{-0.5mm}a_{4,1}}\ar[rr]_-{\hspace{-0.5mm}a_{4,2}} && a_{2,4}a_{4,2}=0\\ 
		&  &a_{2,4}a_{4,1}=0 & & & \\
}}\]

\[{\xymatrix@C-0.5em@R-0.5em{
		& & &    \ul{a_{4,4}}\ar[dr]_-{\hspace{-0.5mm}a_{4,1}}\ar[dl]_-{\hspace{-0.5mm}a_{4,2}} & & & \\
		&(a''_{4,4}-12a'_{4,4})a_{4,1}=0 & \ul{a_{4,2}}\ar[dl]_-{a_{2,4}} & &\ul{a_{4,1}}\ar@/^1pc/[r]^-{\hspace{-0.5mm}a_{1,4}} & {\ul{a_{4,1}a_{1,4}}}\ar[d]_-{\hspace{-0.5mm}a_{4,2}}\ar[rr]_-{\hspace{-0.5mm}a_{4,1}} & &a'_{4,4}a_{4,1}=0 \\
		& {\ul{a_{4,2}a_{2,4}}}\ar[d]_-{\hspace{-0.5mm}a_{4,2}}\ar[u]_-{\hspace{-0.5mm}a_{4,1}}& & & & a'_{4,4}a_{4,2}=0  &\\
		& (a''_{4,4}-12a'_{4,4})a_{4,2}=0 & & & & & & & & & \\
}}\]

The multiplication tree of the idempotent $a_{3,3}$ consists only of the element $a_{3,3}$.

So the kernel of $\varphi$ contains the elements:
\begin{center}
	$\begin{array}{cccccccc}
	\pi\rho   &, &\sigma\vartheta &, & \rho\pi \rho&,&\vartheta\sigma \rho~, \\
	\pi\vartheta&, &\sigma\rho &, & \rho\pi \vartheta&,&\vartheta\sigma \vartheta~.\\
	\end{array}$ 
\end{center}

Let $I$ be the ideal in $\Q\Psi$ generated by those elements. So $I\subseteq\kernel(\phi)$. Therefore, $\phi$ induces a surjective $\Q$-algebra morphism from $\Q\Psi/I$ to $A'$ .

Note that $\Q\Psi / I$ is $\Q$-linearly generated by \[\mathcal{N}:=\{\tilde a_{3,3}+I, \tilde a_{2,2}+I,\tilde a_{4,4}+I, \tilde a_{1,1}+I, \sigma+I,\pi+I,\vartheta+I,\rho+I, \vartheta\sigma+I, \rho\pi+I\},\] cf.\ the underlined elements above. 
To see that, note that a product $\xi$ of $k$ generators may be written as a product in $\mathcal{N}$ of $k'$ generators and a product of $k''$ generators, where $k=k'+k''$ and where $k'$ is chosen maximal. We call $k''$ the excess of $\xi$. If $k''\geq 1$ then, using the trees above, we may write $\xi$ as an $\Q$-linear combination of products of generators that have excess $\leq k''-1$. In the present case, we even have $\xi=0$.

Moreover, note that  $|\mathcal{N}|=10=\dim_\Q(A')$. 

Since we have a surjective $\Q$-algebra morphism from $\Q\Psi /I$ to $A'$, this dimension argument shows this morphism to be bijective. In particular, $I=\kernel(\phi)$.

We may reduce this list to obtain $\kernel(\phi)= (\pi\rho, \sigma\vartheta,\pi\vartheta,\sigma\rho)$. So we obtain the 

\begin{Proposition}\label{pfads3}  Recall that $I=  (\pi\rho, \sigma\vartheta,\pi\vartheta,\sigma\rho)$. We have the isomorphism of $\Q$-algebras 
	
	\[\begin{array}{rcl}
	A' &\xrightarrow{\sim}&\Q\left[{\xymatrix@C-0.5em@R-0.5em{
			\tilde a_{3,3} & \tilde a_{2,2}\ar@/^1pc/[rr]^-{\hspace{-0.7mm}\sigma} & &\tilde a_{4,4}\ar@/^1pc/[ll]^-{\hspace{-0.7mm}\vartheta}\ar@/_1pc/[rr]_-{\hspace{-0.7mm}\rho} & &  \tilde a_{1,1}\ar@/_1pc/[ll]_-{\hspace{-0.7mm}\pi}
	}} \right]/I =\Q\Psi /I\\
	a_{1,1}&\mapsto&\tilde a_{1,1}+ I \\
	a_{2,2} &\mapsto&\tilde a_{2,2}+I \\
	a_{3,3} &\mapsto& \tilde a_{3,3}+I \\
	a_{4,4} &\mapsto&\tilde a_{4,4}+I \\
	a_{4,1} &\mapsto& \rho+ I\\
	a_{1,4} &\mapsto& \pi +I\\
	a_{4,2} &\mapsto& \vartheta+I\\
	a_{2,4}&\mapsto& \sigma +I ~.\\
	\end{array}\]
	
	In particular, $\Q\Psi/I$ is Morita equivalent to $A\cong\B_\Q(\sd,\sd)$.
\end{Proposition}

\section[The double Burnside $R$-algebra \texorpdfstring{$\B_R(\sd,\sd)$}{B\_R(S\textthreeinferior,S\textthreeinferior)} for \texorpdfstring{$R\in\{\Z,\Z_{(2)},\F_{2},\Z_{(3)},\F_{3}\}$}{R\in\{Z,Z\_{(2)},F\_2,Z\_(3),F\_3\}}]{The double Burnside $R$-algebra $\B_R(\sd,\sd)$ for\\ $R\in\{\Z,\Z_{(2)},\F_{2},\Z_{(3)},\F_{3}\}$}\label{int}
\subsection{\texorpdfstring{$\B_\Z(\sd,\sd)$}{B\_Z(S\textthreeinferior,S\textthreeinferior)} via congruences}\label{congruences}
Recall that 
$ A=\bigoplus\limits_{i,j\in[1,4]}A_{i,j}\xrightarrow[\gamma]{\sim} \B_\Q$ ,
cf.\ \Cref{gamma}. 
In the $\Q$-algebra $A$, we define the $\Z$-order \begin{center} $A_\Z := \left[\begin{array}{lllllll}
	{A_\Z}_{,1,1} && {A_\Z}_{,1,2} && {A_\Z}_{,1,3} && {A_\Z}_{,1,4} \\
	{A_\Z}_{,2,1} && {A_\Z}_{,2,2} && {A_\Z}_{,2,3} && {A_\Z}_{,2,4} \\
	{A_\Z}_{,3,1} && {A_\Z}_{,3,2} && {A_\Z}_{,3,3} && {A_\Z}_{,3,4} \\
	{A_\Z}_{,4,1} && {A_\Z}_{,4,2} && {A_\Z}_{,4,3} && {A_\Z}_{,4,4} \\
	\end{array}\right]:= \left[\begin{array}{lllllll}
	\Z^{3\ti 3} && 0 && 0 && \Z^{3\ti 1} \\
	\phantom{(}   0 && \Z && 0 && \Z \\
	\phantom{(}   0 && 0 && \Z && 0 \\
	\phantom{(}  \Z^{1\ti 3} && \Z && 0 && \Z[\overline{\eta},\overline{\xi}] 
	\end{array}\right]\subseteq A$ . \end{center}

In fact, $A_\Z$ is a subring of $A$, as $\al_{i,j,k}({{A_\Z}_{,i,j}\ti{A_\Z}_{,j,k}})\subseteq {A_\Z}_{,i,k}$ for $i,j,k\in [1,4]$.

\begin{Remark}\rm\label{maxorder}
	As $A\cong \B_\Q$ is not semisimple, there are no maximal $\Z$-orders in $A$, \cite[\S 10]{Reiner}.
	So $A_\Z$ is not a canonical choice of a $\Z$-order in $A$, but it nonetheless enables us to describe $\Lambda$ inside $A_\Z$ via congruences.
\end{Remark}

Consider the following elements of $\UU(A)$.\[\footnotesize\begin{array}{rccccccccl}
x_1&: =&{\setlength{\arraycolsep}{2pt}\left[\begin{array}{rrrrrrrrrrc}
	0 && -2 && 0  && 0 && 0 && 0 \\
	6 && 6 && -4&& 0 && 0 && 0 \\
	0 && 0 && 1  && 0 && 0 && 0 \\
	0       && 0        &&      0   && 1&& 0 && 0\\ 
	0       && 0        &&      0   && 0 &&   1&&0\\ 
	0&& 0 && 0  && 0  && 0 && 1
	\end{array}\right]}~, &
x_2&: =& {\setlength{\arraycolsep}{2pt}\left[\begin{array}{rrrrrrrrrrc}
	1 && 0 && 0  && 0 && 0 && 0 \\
	0 && 1 && 0&& 0 && 0 && 0 \\
	0 && 7 && 1  && 0 && 0 && 0 \\
	0       && 0        &&      0   && 1&& 0 && 0\\ 
	0       && 0        &&      0   && 0 &&   1&&0\\ 
	0&& 0 && 0  && 0  && 0 && 1
	\end{array}\right]}~, & 
x_3&: =& {\setlength{\arraycolsep}{2pt}\left[\begin{array}{rrrrrrrrrrc}
	1 && 0 && 0  && 0 && 0 && 0 \\
	0 && 1 && 0&& 0 && 0 && 0 \\
	0 && 0 && 1  && 0 && 0 && 1 \\
	0       && 0        &&      0   && 1&& 0 && 0\\ 
	0       && 0        &&      0   && 0 &&   1&&0\\ 
	0&& 0 && 6  && 0  && 0 && 1
	\end{array}\right]}~. 
\end{array}\] We define the injective ring morphism $\delta: \B_\Z\rightarrow A,~ y\mapsto x_3^{-1}\cdot x_2^{-1}\cdot x_1^{-1}\cdot\ga^{-1}(y)\cdot x_1\cdot x_2\cdot x_3~.$ The conjugating element $x_1$ was constructed  such that the its image lies in $A_\Z$. The elements $x_2,~x_3$ serve the purpose of simplifying the congruences of $\delta(\B_\Z)$.

\begin{Theorem}\label{Proplamb}
	The image $\delta(\B_\Z)$ in $A_\Z$ is given by \begin{center}
		\mbox{$\Lambda:=\delta(\B_\Z)={\renewcommand{\arraystretch}{1.0}\setlength{\arraycolsep}{1pt}\left\{\begin{array}{rcl}
				\raisebox{20mm}{${\setlength{\arraycolsep}{2pt}\left[\begin{array}{cccccc}
							s_{1,1} & s_{1,2}  & s_{1,3}  & ~0~ & ~0~ & t_1 \\
							s_{2,1} & s_{2,2}  & s_{2,3}  & ~0~ & ~0~ & t_2 \\
							s_{3,1} & s_{3,2}  & s_{3,3}  & ~0~ & ~0~ & t_3 \\
							0       & 0        &      0   & ~u~& ~0~ & v\\ 
							0       & 0        &      0   &~ 0~ &   ~w~&0\\ 
							x_1& x_2 & x_3  & ~y~  & ~0~ & z_1+ z_2\oeta+ z_3\oxi
						\end{array}\right]}$}&\raisebox{20mm}{$\in A_\Z :$}&\begin{array}{ccccccc} 
				2w-2z_1&\equiv_8 &z_2 &\equiv_4& z_3 &\equiv_4& 0  \\
				x_1 &\equiv_4&0\\
				x_2 &\equiv_4&0\\
				x_3 &\equiv_4&0\\
				y &\equiv_2&0\\
				t_1 &\equiv_2&0\\
				t_2 &\equiv_2&0\\
				t_3 &\equiv_2&0\\
				v &\equiv_2&0\\[0.5cm]
				x_1 &\equiv_3&0\\
				x_2 &\equiv_3&0\\
				x_3 &\equiv_3&0\\
				z_2 &\equiv_3&0\\
				\end{array}\end{array}\right\} }~.$}
	\end{center}
	In particular, we have $\B_\Z=\B_\Z(\sd,\sd)\cong\Lambda$ as rings.
\end{Theorem}
{ More symbolically written, we have 
	
	$\Lambda=\left[
	\vcenter{\xymatrix@C-1.4em@R-1.5em{
			\Z & \Z & \Z & 0 & 0 & (2) &\\
			\Z & \Z & \Z & 0 & 0 & (2) &\\
			\Z & \Z & \Z & 0 & 0 & (2) &\\
			0       & 0        &      0   & \Z & 0 &(2) &\\ 
			0       & 0        &      0   & 0 &   \Z &0& & *++[o][F]{8}\ar@{-}@/_1pc/[lll]_{\hspace{-0.5mm}-2}\ar@{-}@/_1pc/[dll]_{\hspace{-0.5mm}2}\ar@{-}@/_1pc/[dl]^{\hspace{-0.5mm}1}\\ 
			(12) & (12) & (12) &(2) & 0 & \Z &+(12)\oeta &+(4)\oxi &\\
	}}\right] . $}
\begin{proof} We identify $\Z^{22\ti 1}$ and $A_\Z$ along the isomorphism \[{\small\left(\setlength{\arraycolsep}{0.5pt}\begin{array}{c}
		s_{1,1},
		s_{2,1},
		s_{3,1},
		s_{1,2},
		s_{2,2},
		s_{3,2},
		s_{1,3},
		s_{2,3},
		s_{3,3},\\
		x_1,
		x_2,
		x_3,
		u,
		y,
		w,
		t_1,
		t_2,
		t_3,
		v,
		z_1,
		z_2,
		z_3
		\end{array}\right)^{\rm t}}\mapsto\left[\setlength{\arraycolsep}{2pt}\begin{array}{cccccc}
	s_{1,1} & s_{1,2}  & s_{1,3}  & 0 & 0 & t_1 \\
	s_{2,1} & s_{2,2}  & s_{2,3}  & 0 & 0 & t_2 \\
	s_{3,1} & s_{3,2}  & s_{3,3}  & 0 & 0 & t_3 \\
	0       & 0        &      0   & u& 0 & v\\ 
	0       & 0        &      0   & 0 &   w&0\\ 
	x_1& x_2 & x_3  & y  & 0 & z_1+ z_2\oeta+ z_3\oxi
	\end{array}\right].  \]

Let $M$ be the representation matrix of $\delta$, with respect to the bases $\tilde{\mathcal{H}}=(H_{0,0}, 	H_{0,1} ,       H_{1,0}, \linebreak     H^\Delta_1        ,    H_{0,4} ,       H_{4,0} ,   H^\Delta_4,		H_{1,1}      , H_{0,5} , H_{5,0} ,      H_7,H_{1,4}   ,H_{4,1}         ,   H_6       , H^\Delta_5  ,    H_{4,4}   ,    H_{5,1}      ,       H_{1,5}  ,      H_{5,4}  ,~  H_{4,5}    ,  H_8   , H_{5,5}) $ of $\B_\Z$ and the standard basis of $A_\Z\ab$. We obtain
	
	\[\hspace*{-1.5cm}M=\left({\renewcommand{\arraystretch}{1.05}\scriptsize\setlength{\arraycolsep}{1pt}\begin{array}{rrrrrrrrrrrrrrrrrrrrrrrrrrrrrrrrrrrrrrrrrrrr}
		0 &    0 &   15 &   -3 &    0 &   20 &    8 &    6  & \phantom{-}  0 &   25 &    7 &    9 &    8 &   -3  &   \phantom{--}1 &   12 &   10 &    3 &   15 &    4 &    3 &    5\\
		0 &    0 &  -18 &    0 &    0 &  -24 &    0 &   -9  &   0 &  -30 &  -12 &   -6 &  -12 &    0  &   0 &   -8 &  -15 &   -3 &  -10 &   -4 &   -4 &   -5\\
		0 &    0 &  126 &   -6 &    0 &  168 &   12 &   60  &   0 &  210 &   78 &   48 &   80 &   -6  &   0 &   64 &  100 &   21 &   80 &   28 &   26 &   35\\
		-5 &   -2 &  -60 &    9 &   -3 &  -55 &  -23 &  -24  &  -1 &  -85 &  -16 &  -36 &  -22 &   10  &   0 &  -33 &  -34 &  -12 &  -51 &  -11 &   -5 &  -17\\
		6 &    3 &   72 &    3 &    2 &   66 &    2 &   36  &   1 &  102 &   33 &   24 &   33 &    1  &   1 &   22 &   51 &   12 &   34 &   11 &   11 &   17\\
		-42 &  -20 & -504 &    2 &  -16 & -462 &  -46 & -240  &  -7 & -714 & -208 & -192 & -220 &   15  &   0 & -176 & -340 &  -84 & -272 &  -77 &  -65 & -119\\
		0 &    0 &  -10 &    2 &    0 &  -10 &   -4 &   -4  &   0 &  -15 &   -3 &   -6 &   -4 &    2  &   0 &   -6 &   -6 &   -2 &   -9 &   -2 &   -1 &   -3\\
		0 &    0 &   12 &    0 &    0 &   12 &    0 &    6  &   0 &   18 &    6 &    4 &    6 &    0  &   0 &    4 &    9 &    2 &    6 &    2 &    2 &    3\\
		0 &    0 &  -84 &    4 &    0 &  -84 &   -6 &  -40  &   0 & -126 &  -38 &  -32 &  -40 &    4  &   1 &  -32 &  -60 &  -14 &  -48 &  -14 &  -12 &  -21\\
		0 &    0 & -756 &   36 &    0 &-1008 &  -72 & -360  &   0 &-1260 & -468 & -288 & -480 &   72  &   0 & -384 & -600 & -108 & -480 & -144 & -120 & -180\\
		\phantom{-}252 &  \phantom{-}120 & 3024 &  -12 &   96 & 2772 &  276 & 1440  &  36 & 4284 & 1248 & 1152 & 1320 & -228  &   0 & 1056 & 2040 &  432 & 1632 &  396 &  252 &  612\\
		0 &    0 &  504 &  -24 &    0 &  504 &   36 &  240  &   0 &  756 &  228 &  192 &  240 &  -48  &   0 &  192 &  360 &   72 &  288 &   72 &   48 &  108\\
		0 &    0 &    0 &    1 &    0 &    0 &    0 &    0  &   0 &    0 &    1 &    0 &    0 &    1  &   1 &    0 &    0 &    0 &    0 &    0 &    1 &    0\\
		0 &    0 &    0 &    0 &    0 &    0 &    0 &    0  &   0 &    0 &    0 &    0 &    0 &   -2  &   0 &    0 &    0 &    0 &    0 &    0 &   -2 &    0\\
		0 &    0 &    0 &    0 &    0 &    0 &    2 &    0  &   0 &    0 &    0 &    0 &    0 &    0  &   1 &    0 &    0 &    0 &    0 &    0 &    0 &    0\\
		0 &    0 &  -10 &    2 &    0 &  -10 &   -4 &   -4  &   0 &  -10 &   -4 &   -6 &   -4 &    2  &   0 &   -6 &   -4 &   -2 &   -6 &   -2 &   -2 &   -2\\
		0 &    0 &   12 &    0 &    0 &   12 &    0 &    6  &   0 &   12 &    6 &    4 &    6 &    0  &   0 &    4 &    6 &    2 &    4 &    2 &    2 &    2\\
		0 &    0 &  -84 &    4 &    0 &  -84 &   -6 &  -40  &   0 &  -84 &  -40 &  -32 &  -40 &    4  &   0 &  -32 &  -40 &  -14 &  -32 &  -14 &  -14 &  -14\\
		0 &    0 &    0 &    0 &    0 &    0 &    0 &    0  &   0 &    0 &   -2 &    0 &    0 &    0  &   0 &    0 &    0 &    0 &    0 &    0 &   -2 &    0\\
		0 &    0 &    0 &    0 &    0 &    0 &    0 &    0  &   0 &    0 &    0 &    0 &    0 &    0  &   1 &    0 &    0 &    0 &    0 &    0 &    0 &    0\\
		0 &    0 &  504 &  -24 &    0 &  504 &   36 &  240  &   0 &  504 &  240 &  192 &  240 &  -48  &   0 &  192 &  240 &   72 &  192 &   72 &   24 &   72\\
		0 &    0 &    0 &    0 &    0 &    0 &    0 &    0  &   0 &    0 &    0 &    0 &    0 &    0  &   0 &    0 &    0 &    0 &    0 &    0 &    4 &    0\\
		\end{array}}\right)~.\]
	
	Let $\lambda:=\scriptsize\left[\begin{array}{cccccc}
	s_{1,1} & s_{1,2}  & s_{1,3}  & 0 & 0 & t_1 \\
	s_{2,1} & s_{2,2}  & s_{2,3}  & 0 & 0 & t_2 \\
	s_{3,1} & s_{3,2}  & s_{3,3}  & 0 & 0 & t_3 \\
	0       & 0        &      0   & u& 0 & v\\ 
	0       & 0        &      0   & 0 &   w&0\\ 
	x_1& x_2 & x_3  & y  & 0 & z_1+z_2\oeta+z_3\oxi
	\end{array}\right]\in A_\Z\ab$, identified with $\lambda\in \Z^{22\ti 1}$.

We have 
	\[\footnotesize
\begin{array}{rcl}
	\lambda\in\Lambda &\Leftrightarrow & \exists~q\in\Z^{22\ti 1} \text{ such that }  \lambda=Mq\\ &\Leftrightarrow&  \exists~q\in\Z^{22\ti 1} \text{ such that }  M^{-1}\cdot\lambda =q\\  &\Leftrightarrow& 24M^{-1}\cdot\lambda
	\in 24\Z^{22\times 1} \\
	&\Leftrightarrow& 
	{\setlength{\arraycolsep}{1mm}
		{\scriptsize\left( 
			\begin{array} {rrrrrrrrrrrrrrrrrrrrrrrrrrrrrrrrrrrrrrrrrrrrrr}
			
			0& 0&  0 & 0 & 0&  0 & 0&  0 & 0 & 2 & 0 & 0 & 0 & 0&  0&  0 & 0&  0 & 0&  0&  0 & 0\\
			0& 0&  0 & 0 & 0&  0 & 0&  0 & 0 & 0 & 2 & 0 & 0 & 0&  0&  0 & 0&  0 & 0&  0&  0 & 0\\
			0& 0&  0 & 0 & 0&  0 & 0&  0 & 0 & 0 & 0 & 2 & 0 & 0&  0&  0 & 0&  0 & 0&  0&  0 & 0\\
			0& 0&  0 & 0 & 0&  0 & 0&  0 & 0 & 0 & 0 & 0 & 0 &12&  0&  0 & 0&  0 & 0&  0&  0 & 0\\
			0& 0&  0 & 0 & 0&  0 & 0&  0 & 0 & 0 & 0 & 0 & 0 & 0&  6&  0 & 0&  0 & 0& 18&  1 & 0\\
			0& 0&  0 & 0 & 0&  0 & 0&  0 & 0 & 0 & 0 & 0 & 0 & 0&  0& 12 & 0&  0 & 0&  0&  0 & 0\\
			0& 0&  0 & 0 & 0&  0 & 0&  0 & 0 & 0 & 0 & 0 & 0 & 0&  0&  0 &12&  0 & 0&  0&  0 & 0\\
			0& 0&  0 & 0 & 0&  0 & 0&  0 & 0 & 0 & 0 & 0 & 0 & 0&  0&  0 & 0& 12 & 0&  0&  0 & 0\\
			0& 0&  0 & 0 & 0&  0 & 0&  0 & 0 & 0 & 0 & 0 & 0 & 0&  0&  0 & 0&  0 &12&  0&  0 & 0\\
			0& 0&  0 & 0 & 0&  0 & 0&  0 & 0 & 0 & 0 & 0 & 0 & 0&  0&  0 & 0&  0 & 0&  0&  2 & 0\\
			0& 0&  0 & 0 & 0&  0 & 0&  0 & 0 & 0 & 0 & 0 & 0 & 0&  0&  0 & 0&  0 & 0&  0&  0 & 6\\
			
			\end{array}\right)
	}}
	
	\renewcommand{\arraystretch}{0.6}
	{\left( \begin{array}{c}
		s_{1,1}\\
		s_{2,1}\\
		s_{3,1}\\
		\smash{\vdots}\rule[4.5mm]{0mm}{0mm}\\
		s_{3,3}\\
		x_1\\
		x_2\\
		x_3\\
		u\\
		y\\
		w\\
		t_1\\
		t_2\\
		t_3\\
		v\\
		z_1\\
		z_2\\
		z_3
		\end{array}\right)}
	\in 24\Z^{11\times 1}\\
	& \Leftrightarrow &{\renewcommand{\arraystretch}{0.9}\setlength{\arraycolsep}{2pt}\left\{\begin{array}{ccccccc}
		2w-2z_1&\equiv_8 &z_2 &\equiv_4& z_3 &\equiv_4& 0 \\
		x_1 &\equiv_4&0\\
		x_2 &\equiv_4&0\\
		x_3 &\equiv_4&0\\
		y &\equiv_2&0\\
		t_1 &\equiv_2&0\\
		t_2 &\equiv_2&0\\
		t_3 &\equiv_2&0\\
		v &\equiv_2&0\\\\
		x_1 &\equiv_3&0\\
		x_2 &\equiv_3&0\\
		x_3 &\equiv_3&0\\
		z_2 &\equiv_3&0\\
		\end{array}\right\} }~.
\end{array}\]
\end{proof}

\subsection{Localisation at 2: \texorpdfstring{$\B_{\Z_{(2)}}(\sd,\sd)$}{B\_Z\_(2)(S\textthreeinferior,S\textthreeinferior)} via congruences}
Write $R:=\Z_{(2)}$. In the $\Q$-algebra $A$, cf.\ \Cref{gamma}, we have the $R$-order \begin{center} $A_R := \left[\begin{array}{lllllll}
	{A_R}_{,1,1} && {A_R}_{,1,2} && {A_R}_{,1,3} && {A_R}_{,1,4} \\
	{A_R}_{,2,1} && {A_R}_{,2,2} && {A_R}_{,2,3} && {A_R}_{,2,4} \\
	{A_R}_{,3,1} && {A_R}_{,3,2} && {A_R}_{,3,3} && {A_R}_{,3,4} \\
	{A_R}_{,4,1} && {A_R}_{,4,2} && {A_R}_{,4,3} && {A_R}_{,4,4} \\
	\end{array}\right]:= \left[\begin{array}{lllllll}
	R^{3\ti 3} && 0 && 0 && R^{3\ti 1} \\
	\phantom{(}   0 && R && 0 && R \\
	\phantom{(}   0 && 0 && R && 0 \\
	\phantom{(}  R^{1\ti 3} && R && 0 && R[\overline{\eta},\overline{\xi}] 
	\end{array}\right]\subseteq A~.$ \end{center}

\begin{Corollary}\label{congrusd2}
	We have
	\[ \mbox{$\Lambda_{(2)}={\renewcommand{\arraystretch}{1.0}\setlength{\arraycolsep}{1pt}\left\{\begin{array}{rcl}
			\raisebox{8mm}{${\setlength{\arraycolsep}{2pt}\left[\begin{array}{cccccc}
					s_{1,1} & s_{1,2}  & s_{1,3}  & ~0~ & ~0~ & t_1 \\
					s_{2,1} & s_{2,2}  & s_{2,3}  & ~0~ & ~0~ & t_2 \\
					s_{3,1} & s_{3,2}  & s_{3,3}  & ~0~ & ~0~ & t_3 \\
					0       & 0        &      0   & ~u~& ~0~ & v\\ 
					0       & 0        &      0   &~ 0~ &   ~w~&0\\ 
					x_1& x_2 & x_3  & ~y~  & ~0~ & z_1+ z_2\oeta+ z_3\oxi
					\end{array}\right]}$}&\raisebox{8mm}{$\in A_R :$}&\begin{array}{ccccccc} 
			2w-2z_1&\equiv_8 &z_2 &\equiv_4& z_3 &\equiv_4& 0  \\
			x_1 &\equiv_4&0\\
			x_2 &\equiv_4&0\\
			x_3 &\equiv_4&0\\
			y &\equiv_2&0\\
			t_1 &\equiv_2&0\\
			t_2 &\equiv_2&0\\
			t_3 &\equiv_2&0\\
			v &\equiv_2&0	
			\end{array}\end{array}\right\} }$}\subseteq A_{R}\ab.\] In particular, we have $\B_{R}=\B_{R}(\sd,\sd)\cong\Lambda_{(2)}$ as $R$-algebras.
\end{Corollary}
More symbolically written, we have 

\[\Lambda_{(2)}=\left[
\vcenter{\xymatrix@C-1.4em@R-1.5em{
		R & R & R & 0 & 0 & (2) &\\
		R & R & R & 0 & 0 & (2) &\\
		R & R & R & 0 & 0 & (2) &\\
		0       & 0        &      0   & R & 0 &(2) &\\ 
		0       & 0        &      0   & 0 &   R &0& & *++[o][F]{8}\ar@{-}@/_1pc/[lll]_{\hspace{-0.5mm}2}\ar@{-}@/_1pc/[dll]_{\hspace{-0.5mm}-2}\ar@{-}@/_1pc/[dl]^{\hspace{-0.5mm}1}\\ 
		(4) & (4) & (4) &(2) & 0 & R &~+(4)\oeta &~+(4)\oxi &\\
}}\right]. \]
\begin{Remark}
	We claim that $1_{\Lambda_{(2)}} =e_1 +e_2 +e_3 +e_4 +e_5$ is an orthogonal decomposition into primitive idempotents, where \[ \footnotesize\begin{array}{rcl c rcl c rcl}
	e_1&: =& {\setlength{\arraycolsep}{2pt}\left[\begin{array}{rrrrrrrrrrc}
		1 && 0 && 0  && 0 && 0 && 0 \\
		0 && 0 && 0&& 0 && 0 && 0 \\
		0 && 0 && 0  && 0 && 0 && 0 \\
		0       && 0        &&      0   && 0&& 0 && 0\\ 
		0       && 0        &&      0   && 0 &&   0&&0\\ 
		0&& 0 && 0  && 0  && 0 && 0
		\end{array}\right]} ~, &
	e_2&: =& {\setlength{\arraycolsep}{2pt}\left[\begin{array}{rrrrrrrrrrc}
		0 && 0 && 0  && 0 && 0 && 0 \\
		0 && 1 && 0&& 0 && 0 && 0 \\
		0 && 0 && 0  && 0 && 0 && 0 \\
		0       && 0        &&      0   && 0&& 0 && 0\\ 
		0       && 0        &&      0   && 0 &&   0&&0\\ 
		0&& 0 && 0  && 0  && 0 && 0
		\end{array}\right]}~, 
	&e_3&: =& {\setlength{\arraycolsep}{2pt}\left[\begin{array}{rrrrrrrrrrc}
		0 && 0 && 0  && 0 && 0 && 0 \\
		0 && 0&& 0&& 0 && 0 && 0 \\
		0 && 0 && 1  && 0 && 0 && 0 \\
		0       && 0        &&      0   && 0&& 0 && 0\\ 
		0       && 0        &&      0   && 0 &&   0&&0\\ 
		0&& 0 && 0  && 0  && 0 && 0
		\end{array}\right]}~, \\
	e_4&: =& {\setlength{\arraycolsep}{2pt}\left[\begin{array}{rrrrrrrrrrc}
		0 && 0 && 0  && 0 && 0 && 0 \\
		0 && 0 && 0&& 0 && 0 && 0 \\
		0 && 0 && 0  && 0 && 0 && 0 \\
		0       && 0        &&      0   && 1&& 0 && 0\\ 
		0       && 0        &&      0   && 0 &&   0&&0\\ 
		0&& 0 && 0  && 0  && 0 && 0
		\end{array}\right]}~, &
	e_5&: =& {\setlength{\arraycolsep}{2pt}\left[\begin{array}{rrrrrrrrrrc}
		0 && 0 && 0  && 0 && 0 && 0 \\
		0 && 0 && 0&& 0 && 0 && 0 \\
		0 && 0 && 0  && 0 && 0 && 0 \\
		0       && 0        &&      0   && 0&& 0 && 0\\ 
		0       && 0        &&      0   && 0 &&   1&&0\\ 
		0&& 0 && 0  && 0  && 0 && 1
		\end{array}\right]}~.\end{array}\]\end{Remark}
\begin{proof}
	We have 
	$\e_1\Lambda_{(2)} \e_1\cong R $,
	$\e_2\Lambda_{(2)} \e_2\cong R $,
	$\e_3\Lambda_{(2)} \e_3\cong R $ and
	$\e_4\Lambda_{(2)} \e_4\cong R $. So, it follows that $\e_1,\e_2,\e_3,\e_4$ are primitive. 
	
	As $R$-algebras, we have \[\begin{array}{rcl}\e_5\Lambda_{(2)} \e_5 &\cong&\left\{\left(\begin{array}{cc}
	w,   & z_1+z_2\oeta+z_3\oxi
	\end{array}\right)\in R\ti R[\oeta,\oxi]: 2w-2z_1\equiv_8 z_2 \equiv_4 z_3 \equiv_4 0\right\}=:\Gamma\\\\
	&\subseteq&  R\ti R[\oeta,\oxi]~ .\end{array}\]
	
	To show that $\e_5$ is primitive, we show that $\Gamma$ is local. 
	
	We have the $R$-linear basis $(b_1,b_2,b_3,b_4)$ of $ \Gamma$, where 
	\begin{center}
		$\begin{array}{rclccl}
		b_1&=& \left(\begin{array}{cc}
		1,
		& 1
		\end{array}\right), &
		b_2&=& \left(\begin{array}{cc}
		0,
		& 2+4\oeta
		\end{array}\right), \\\\b_3&=& \left(\begin{array}{cc}
		0,
		& 8\oeta
		\end{array}\right), &b_4&=& \left(\begin{array}{cc}
		0,
		& 4\oxi
		\end{array}\right)~. 
		\end{array}$ 
	\end{center}
	
	We {\it claim} that the Jacobson radical of $\Gamma$ is given by $ J:=~_R\langle 2b_1,b_2,b_3,b_4\rangle$, that $\Gamma/J\cong\F_2$ and that $\Gamma$ is local.
	
	In fact, the multiplication table for the basis elements is given by 
	
	\begin{center}
		$\begin{array}{c|c|c|c|c}
		(\cdot) & b_1 & b_2 & b_ 3 &b_4 \\\hline
		b_1 & b_1 & b_2 & b_3 & b_4 \\\hline
		b_2 & b_2 & 2b_2+b_3 & 2b_3 & 2 b_4\\\hline
		b_3 & b_3 & 2b_3 & 0 & 0\\\hline
		b_4 & b_4 & 2b_4 & 0 & 0\\
		\end{array}$~.
	\end{center}
	This shows that $J$ is an ideal. Moreover,  $J$ is topologically nilpotent as \begin{center}$J^3={_R\langle} 8b_1,4b_2,2b_3,4b_4\rangle\subseteq 2\e_5\Lambda_{(2)} \e_5 $ . \end{center} Since $\Gamma / J \cong \F_2$, the {\it claim} follows. \end{proof}

\subsection{\texorpdfstring{$\B_{\Z_{(2)}}(\sd,\sd)$}{B\_Z\_(2)(S\textthreeinferior,S\textthreeinferior)} and \texorpdfstring{$\B_{\F_{2}}(\sd,\sd)$}{B\_F\_2(S\textthreeinferior,S\textthreeinferior)} as path algebras modulo relations} Write $R :=\Z_{(2)}.$
We aim to write $\Lambda_{(2)}$, up to Morita equivalence, as path algebra modulo relations.  The $R$-algebra $\Lambda_{(2)}$ is Morita equivalent to $\Lambda'_{(2)} := (\e_3 +\e_4 +\e_5)\Lambda_{(2)}(\e_3 +\e_4 +\e_5)$ since \mbox{$\Lambda_{(2)}\e_1\cong \Lambda_{(2)}\e_2\cong\Lambda_{(2)}\e_3$} using multiplication with elements of $\Lambda_{(2)}$ with a single nonzero entry 1 in the upper $(3\ti 3)$-corner.

We have the $R$-linear basis of $\Lambda'_{(2)}$ consisting of \[\scriptsize
\begin{array}{rclcrclcrclcrcl}
\e_3&: =& {\setlength{\arraycolsep}{2pt}\left[\begin{array}{rrrrrrrrrrc}
	0 && 0 && 0  && 0 && 0 && 0 \\
	0 && 0&& 0&& 0 && 0 && 0 \\
	0 && 0 && 1  && 0 && 0 && 0 \\
	0       && 0        &&      0   && 0&& 0 && 0\\ 
	0       && 0        &&      0   && 0 &&   0&&0\\ 
	0&& 0 && 0  && 0  && 0 && 0
	\end{array}\right]}~, &
\e_4&: =& {\setlength{\arraycolsep}{2pt}\left[\begin{array}{rrrrrrrrrrc}
	0 && 0 && 0  && 0 && 0 && 0 \\
	0 && 0 && 0&& 0 && 0 && 0 \\
	0 && 0 && 0  && 0 && 0 && 0 \\
	0       && 0        &&      0   && 1&& 0 && 0\\ 
	0       && 0        &&      0   && 0 &&   0&&0\\ 
	0&& 0 && 0  && 0  && 0 &&0
	\end{array}\right]}~,  &
\e_5&: =& {\setlength{\arraycolsep}{2pt}\left[\begin{array}{rrrrrrrrrrc}
	0 && 0 && 0  && 0 && 0 && 0 \\
	0 && 0 && 0&& 0 && 0 && 0 \\
	0 && 0 && 0  && 0 && 0 && 0 \\
	0       && 0        &&      0   && 0&& 0 && 0\\ 
	0       && 0        &&      0   && 0 &&   1&&0\\ 
	0&& 0 && 0  && 0  && 0 && 1
	\end{array}\right]}~,\\\\
\tau_1&: =&  {\setlength{\arraycolsep}{2pt}\left[\begin{array}{rrrrrrrrrrc}
	0 && 0 && 0  && 0 && 0 && 0 \\
	0 && 0 && 0&& 0 && 0 && 0 \\
	0 && 0 && 0  && 0 && 0 && 0 \\
	0       && 0        &&      0   && 0&& 0 && 0\\ 
	0       && 0        &&      0   && 0 &&   0&&0\\ 
	0&& 0 && 4  && 0  && 0 && 0
	\end{array}\right]}~, &
\tau_2&: =&{\setlength{\arraycolsep}{2pt}\left[\begin{array}{rrrrrrrrrrc}
	0 && 0 && 0  && 0 && 0 && 0 \\
	0 && 0 && 0&& 0 && 0 && 0 \\
	0 && 0 && 0  && 0 && 0 && 2 \\
	0       && 0        &&      0   && 0&& 0 && 0\\ 
	0       && 0        &&      0   && 0 &&   0&&0\\ 
	0&& 0 && 0  && 0  && 0 && 0
	\end{array}\right]}~, &
\tau_3&: =&{\setlength{\arraycolsep}{2pt}\left[\begin{array}{rrrrrrrrrrc}
	0 && 0 && 0  && 0 && 0 && 0 \\
	0 && 0 && 0&& 0 && 0 && 0 \\
	0 && 0 && 0  && 0 && 0 && 0 \\
	0       && 0        &&      0   && 0&& 0 && 0\\ 
	0       && 0        &&      0   && 0 &&   0&&0\\ 
	0&& 0 && 0  && 2  && 0 &&0
	\end{array}\right]}~,\\\\
\tau_4&: =& {\setlength{\arraycolsep}{2pt}\left[\begin{array}{rrrrrrrrrrc}
	0 && 0 && 0  && 0 && 0 && 0 \\
	0 && 0 && 0&& 0 && 0 && 0 \\
	0 && 0 && 0  && 0 && 0 && 0 \\
	0       && 0        &&      0   && 0&& 0 && 2\\ 
	0       && 0        &&      0   && 0 &&   0&&0\\ 
	0&& 0 && 0  && 0  && 0 && 0
	\end{array}\right]}~, &
\tau_5&: =& {\setlength{\arraycolsep}{2pt}\left[\begin{array}{rrrrrrrrrrc}
	0 && 0 && 0  && 0 && 0 && 0 \\
	0 && 0 && 0&& 0 && 0 && 0 \\
	0 && 0 && 0  && 0 && 0 && 0 \\
	0       && 0        &&      0   && 0&& 0 && 0\\ 
	0       && 0        &&      0   && 0 &&   0&&0\\ 
	0&& 0 && 0  && 0  && 0 && 8\oeta
	\end{array}\right]}~,&
\tau_6&: =&{\setlength{\arraycolsep}{2pt}\left[\begin{array}{rrrrrrrrrrc}
	0 && 0 && 0  && 0 && 0 && 0 \\
	0 && 0 && 0&& 0 && 0 && 0 \\
	0 && 0 && 0  && 0 && 0 && 0 \\
	0       && 0        &&      0   && 0&& 0 && 0\\ 
	0       && 0        &&      0   && 0 &&   0&&0\\ 
	0&& 0 && 0  && 0  && 0 && 4\oxi
	\end{array}\right]}~, \\\\
\tau_7&: =& {\setlength{\arraycolsep}{2pt}\left[\begin{array}{rrrrrrrrrrc}
	0 && 0 && 0  && 0 && 0 && 0 \\
	0 && 0 && 0&& 0 && 0 && 0 \\
	0 && 0 && 0  && 0 && 0 && 0 \\
	0       && 0        &&      0   && 0&& 0 && 0\\ 
	0       && 0        &&      0   && 0 &&   0&&0\\ 
	0&& 0 && 0  && 0  && 0 && 2+4\oeta
	\end{array}\right]}
\end{array}\]

We have $\tau_5 = \tau_1\tau_2$ and $\tau_6 = \tau_3\tau_4+ 6\tau_1\tau_2$. Hence, as an $R$-algebra 
$\Lambda'_{(2)}$ is generated by $\e_3, \e_4, \e_5, \tau_1, \tau_2, \tau_3 ,\tau_4, \tau_7$.

Consider the quiver $\Psi:=\left[
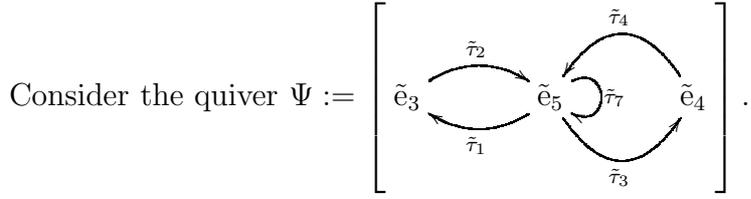
{\xymatrix@C-0.7em@R-0.7em{
		\tilde \e_3\ar@/^1pc/[rr]^-{\hspace{-0.7mm}\tilde \tau_2} & &\tilde \e_5\ar@(ur,dr)^-{\hspace{-0.7mm}\tilde\tau_7}\ar@/^1pc/[ll]^-{\hspace{-0.7mm}\tilde \tau_1}\ar@/_2pc/[rr]_-{\hspace{-0.7mm}\tilde \tau_3} & &  \tilde \e_4\ar@/_2pc/[ll]_-{\hspace{-0.7mm}\tilde\tau_4}
}}\right].$

We have a surjective $R$-algebra morphism $\phi: R\Psi\rightarrow \Lambda'_{(2)}$ by sending 
\[\begin{array}{rclcrclcrclcrclcrcl}
\tilde \e_3 &\mapsto& \e_3 &,&
\tilde \e_4 &\mapsto& \e_4 &,&
\tilde \e_5 &\mapsto& \e_5 &,&
\tilde \tau_1 &\mapsto& \tau_1 &,&\\
\tilde \tau_2 &\mapsto& \tau_2 &,&
\tilde \tau_3 &\mapsto& \tau_3 &,&
\tilde \tau_4 &\mapsto& \tau_4 &,&
\tilde \tau_7 &\mapsto& \tau_7 &.&
\end{array}\]

We establish the following multiplication trees, where we underline the elements that are not in an $R$-linear relation with previous elements.

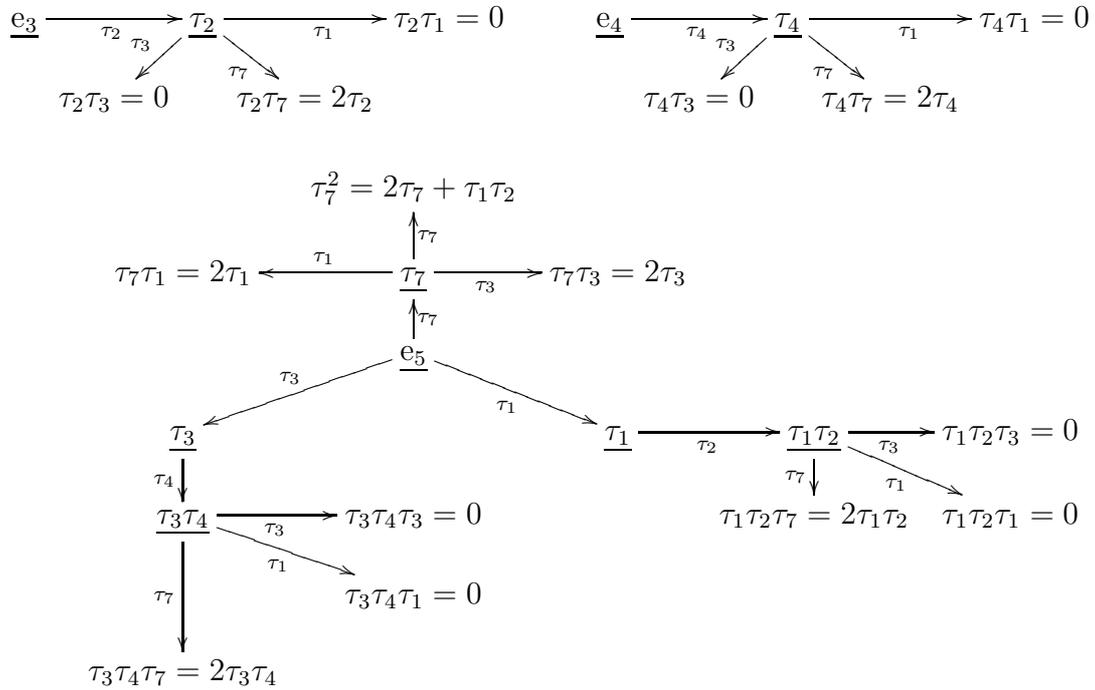
\[{\xymatrix@C-1.9em@R-0.9em{
		\ul{\e_3}\ar[rr]_{\hspace{-0.5mm}\tau_2}& & \ul{\tau_2}\ar[ld]_{\hspace{-0.5mm}\tau_3}\ar[rd]_{\hspace{-0.5mm}\tau_7}\ar[rr]_{\hspace{-0.5mm}\tau_1} && \tau_2\tau_1=0\\ 
		&\tau_2\tau_3 = 0 & &\tau_2\tau_7= 2\tau_2 & \\
}}\hspace*{1cm}{\xymatrix@C-1.9em@R-0.9em{
		\ul{\e_4}\ar[rr]_{\hspace{-0.5mm}\tau_4}& & \ul{\tau_4}\ar[ld]_{\hspace{-0.5mm}\tau_3}\ar[rd]_{\hspace{-0.5mm}\tau_7}\ar[rr]_{\hspace{-0.5mm}\tau_1} && \tau_4\tau_1=0\\ 
		&\tau_4\tau_3 = 0 & &\tau_4\tau_7= 2\tau_4 & \\
}}\]

\[{\xymatrix@C-1.5em@R-0.9em{
		& & & \tau_7^2 = 2\tau_7+\tau_1\tau_2  & & & \\
		& &\tau_7\tau_1= 2\tau_1 & \ul{\tau_7}\ar[u]_-{\hspace{-0.5mm}\tau_7}\ar[r]_-{\hspace{-0.5mm}\tau_3}\ar[l]_-{\hspace{-0.5mm}\tau_1}  &\tau_7\tau_3= 2\tau_3 & & \\
		& & &    \ul{\e_5}\ar[u]_-{\hspace{-0.5mm}\tau_7}\ar[dr]_-{\hspace{-0.5mm}\tau_1}\ar[dl]_-{\hspace{-0.5mm}\tau_3} & & & \\
		& & \ul{\tau_3}\ar[d]_-{\hspace{-0.5mm}\tau_4}&   &\ul{\tau_1}\ar[r]_-{\hspace{-0.5mm}\tau_2} &\ul{\tau_1\tau_2}\ar[r]_-{\hspace{-0.5mm}\tau_3}\ar[d]_-{\hspace{-0.5mm}\tau_7}\ar[rd]_-{\hspace{-0.5mm}\tau_1} &\tau_1\tau_2\tau_3=0 \\
		& &\ul{\tau_3\tau_4}\ar[r]_-{\hspace{-0.5mm}\tau_3}\ar[dd]_-{\hspace{-0.5mm}\tau_7}\ar[dr]_-{\hspace{-0.5mm}\tau_1}  &\tau_3\tau_4\tau_3=0   & &\tau_1\tau_2\tau_7=2\tau_1\tau_2 & \tau_1\tau_2\tau_1=0\\
		& &  &\tau_3\tau_4\tau_1=0     & & & \\
		& &\tau_3\tau_4\tau_7=2\tau_3\tau_4  &   & & & \\
}}\]
So, the kernel of $\varphi$ contains the elements:
\begin{center}
	$\begin{array}{ccccccccccc}
	\ttau_2\ttau_1   &,&\ttau_4\ttau_1 &,& \ttau_1\ttau_2\ttau_1&,&\ttau_3\ttau_4\ttau_1 &,& \ttau_7\ttau_1-2\ttau_1,  \\
	\ttau_2\ttau_3&,&\ttau_4\ttau_3  &,& \ttau_1\ttau_2\ttau_3&,&\ttau_3\ttau_4\ttau_3&,& \ttau_7\ttau_3-2\ttau_3, \\
	\ttau_2\ttau_7-2\ttau_2&,& \ttau_4\ttau_7-2\ttau_4&,&\ttau_1\ttau_2\ttau_7-2\ttau_1\ttau_2 &,&  \ttau_3\ttau_4\ttau_7-2\ttau_3\ttau_4 &,&\ttau_7^2-2\ttau_7-\ttau_1\ttau_2~. \\
	\end{array}$ 
\end{center}
Let $I$ be the ideal generated by these elements. So $I\subseteq\kernel(\phi).$  Therefore, $\phi$ induces a surjective $R$-algebra morphism from $R\Psi/I$ to $\Lambda_{(2)}'\ab$. We may reduce the list of generators to obtain \[I= (
\ttau_2\ttau_1, \ttau_4\ttau_1,\ttau_7\ttau_1-2\ttau_1,\ttau_2\ttau_3, \ttau_4\ttau_3  , \ttau_7\ttau_3-2\ttau_3,
\ttau_2\ttau_7-2\ttau_2, \ttau_4\ttau_7-2\ttau_4,  \ttau_7^2-2\ttau_7-\ttau_1\ttau_2)~.\]

Note that $R\Psi / I$ is $R$-linearly generated by \begin{center} $\mathcal{N}:=\{\tilde \e_3+I, \tilde \e_4+I,\tilde \e_5+I, \ttau_1+I , \ttau_2+I, \ttau_3+I,\ttau_4+I, \ttau_7+I, \ttau_3\ttau_4+I, \ttau_1\ttau_2+I\}$,\end{center}cf.\ the underlined elements above. To see that, note that a product $\xi$ of $k$ generators may be written as a product in $\mathcal{N}$ of $k'$ generators and a product of $k''$ generators, where $k=k'+k''$ and where $k'$ is choosen maximal. We call $k''$ the excess of $\xi$. If $k''\geq 1$ then, using the trees above, we may write $\xi$ as an $R$-linear combination of products of generators that have excess $\leq k''-1$.  Moreover, note that $|\mathcal{N}|=10=\rk_R(\Lambda'_{(2)})$. 

Since we have a surjective $R$-algebra morphism from $R\Psi /I$ to $\Lambda'_{(2)}\ab$, this rank argument shows this morphism to be bijective. In particular, $I=\kernel(\phi)$.

So, we obtain the
\begin{Proposition}\label{pfadsdZ2} Recall that $I=\left(\begin{array}{ccccccccccc}
	\ttau_2\ttau_1   &,& \ttau_2\ttau_3  &,& \ttau_2\ttau_7-2\ttau_2 , \\
	\ttau_4\ttau_1&,& \ttau_4\ttau_3  &,& \ttau_4\ttau_7-2\ttau_4, \\
	\ttau_7\ttau_1-2\ttau_1&,& \ttau_7\ttau_3-2\ttau_3&,&  \ttau_7^2-2\ttau_7-\ttau_1\ttau_2~ \\
	\end{array}\right) $.
	
	We have the isomorphism of $\Z_{(2)}$-algebras
	\begin{center} $\begin{array}{rcl}\Lambda'_{(2)}&\xrightarrow{\sim}& R\left[{\xymatrix@C-0.5em@R-0.5em{
				\tilde \e_3\ar@/^1pc/[rr]^-{\hspace{-0.7mm}\tilde \tau_2} & &\tilde \e_5\ar@(ur,dr)^-{\hspace{-0.7mm}\tilde\tau_7}\ar@/^1pc/[ll]^-{\hspace{-0.7mm}\tilde \tau_1}\ar@/_2pc/[rr]_-{\hspace{-0.7mm}\tilde \tau_3} & &  \tilde \e_4\ar@/_2pc/[ll]_-{\hspace{-0.7mm}\tilde\tau_4}}}\right] /I\\\\
		\e_i&\mapsto& \tilde \e_i+I \text{ for } i\in [3,5]\\
		\tau_j &\mapsto&\ttau_j+I \text{ for } j\in [1,7]\setminus\{5,6\}~.\end{array}$ \end{center} Recall that $\B_{\Z_{(2)}}(\sd,\sd)$ is Morita equivalent to $\Lambda'_{(2)}\ab$.
\end{Proposition}

\begin{Corollary}\label{pfadsdF2}
	As $\F_2$-algebras, we have \[\Lambda'_{(2)}/2\Lambda'_{(2)}~\cong~\F_2\left[{\xymatrix@C-0.5em@R-0.5em{
			\tilde \e_3\ar@/^1pc/[rr]^-{\hspace{-0.7mm}\tilde \tau_2} & &\tilde \e_5\ar@(ur,dr)^-{\hspace{-0.7mm}\tilde\tau_7}\ar@/^1pc/[ll]^-{\hspace{-0.7mm}\tilde \tau_1}\ar@/_2pc/[rr]_-{\hspace{-0.7mm}\tilde \tau_3} & &  \tilde \e_4\ar@/_2pc/[ll]_-{\hspace{-0.7mm}\tilde\tau_4}}}\right] /\left(\begin{array}{ccccccccccc}
	\ttau_2\ttau_1   &,& \ttau_2\ttau_3  &,& \ttau_2\ttau_7,  \\
	\ttau_4\ttau_1&,& \ttau_4\ttau_3  &,& \ttau_4\ttau_7, \\
	\ttau_7\ttau_1&,& \ttau_7\ttau_3&,&  \ttau_7^2-\ttau_1\ttau_2~ \\
	\end{array}\right) ~ .\] Recall that $\B_{\F_{2}}(\sd,\sd)$ is Morita equivalent to $\Lambda'_{(2)}/2\Lambda'_{(2)}\ab$.
\end{Corollary}

\subsection{Localisation at 3: \texorpdfstring{$\B_{\Z_{(3)}}(\sd,\sd)$}{B\_Z\_(3)(S\textthreeinferior,S\textthreeinferior)} via congruences}
Write $R=\Z_{(3)}$. In the $\Q$-algebra $A$, cf.\ \Cref{gamma}, we have the $R$-order \begin{center} $A_R := \left[\begin{array}{lllllll}
	{A_R}_{,1,1} && {A_R}_{,1,2} && {A_R}_{,1,3} && {A_R}_{,1,4} \\
	{A_R}_{,2,1} && {A_R}_{,2,2} && {A_R}_{,2,3} && {A_R}_{,2,4} \\
	{A_R}_{,3,1} && {A_R}_{,3,2} && {A_R}_{,3,3} && {A_R}_{,3,4} \\
	{A_R}_{,4,1} && {A_R}_{,4,2} && {A_R}_{,4,3} && {A_R}_{,4,4} \\
	\end{array}\right]:= \left[\begin{array}{lllllll}
	R^{3\ti 3} && 0 && 0 && R^{3\ti 1} \\
	\phantom{(}   0 && R && 0 && R \\
	\phantom{(}   0 && 0 && R && 0 \\
	\phantom{(}  R^{1\ti 3} && R && 0 && R[\overline{\eta},\overline{\xi}] 
	\end{array}\right]\subseteq A~.$ \end{center}
\begin{Corollary}\label{congruencessd3}
	We have
	\[\Lambda_{(3)}=
	{\setlength{\arraycolsep}{1pt}\left\{\begin{array}{rcccccl}
		{\setlength{\arraycolsep}{1pt}\left[\begin{array}{cccccc}
			s_{1,1} & s_{1,2}  & s_{1,3}  & ~0~ & ~0~ & t_1 \\
			s_{2,1} & s_{2,2}  & s_{2,3}  & ~0~ & ~0~ & t_2 \\
			s_{3,1} & s_{3,2}  & s_{3,3}  & ~0~ & ~0~ & t_3 \\
			0       & 0        &      0   & ~u~& ~0~ & v\\ 
			0       & 0        &      0   & ~0~ &   ~w~&0\\ 
			x_1& x_2 & x_3  & ~y~  & 0 & z_1+z_2\oeta+ z_3\oxi
			\end{array}\right]}\in A_R:
		\begin{array}{ccc}x_1 &\equiv_3&0\\
		x_2 &\equiv_3&0\\
		x_3 &\equiv_3&0\\
		z_2 &\equiv_3&0\end{array}
		\end{array}\right\} }\subseteq A_R\ab.\]
	
	In particular, we have $\B_{R}=\B_{R}(\sd,\sd)\cong\Lambda_{(3)}$ as $R$-algebras.
	
\end{Corollary}
More symbolically written, we have 

$\Lambda_{(3)}=\left[
\vcenter{\xymatrix@C-1.1em@R-1.22em{
		R & R & R & 0 & 0 & R \\
		R & R & R & 0 & 0 & R \\
		R & R & R & 0 & 0 & R \\
		0       & 0        &      0   & R & 0 &R \\ 
		0       & 0        &      0   & 0 &   R &0& \\ 
		(3) & (3) & (3) &R & 0 & R &+(3)\oeta &+R\oxi \\
}}\right]. $

\begin{Remark} We claim that $1_{\Lambda_{(3)}} =\e_1 +\e_2 +\e_3 +\e_4 +\e_5 +\e_6$ is an orthogonal decomposition into primitive idempotents, where \[ \footnotesize\begin{array}{rcl c rcl  crcl}
	\e_1&: =& {\setlength{\arraycolsep}{2pt}\left[\begin{array}{rrrrrrrrrrc}
		1 && 0 && 0  && 0 && 0 && 0 \\
		0 && 0 && 0&& 0 && 0 && 0 \\
		0 && 0 && 0  && 0 && 0 && 0 \\
		0       && 0        &&      0   && 0&& 0 && 0\\ 
		0       && 0        &&      0   && 0 &&   0&&0\\ 
		0&& 0 && 0  && 0  && 0 && 0
		\end{array}\right]} ~, &
	\e_2&: =& {\setlength{\arraycolsep}{2pt}\left[\begin{array}{rrrrrrrrrrc}
		0 && 0 && 0  && 0 && 0 && 0 \\
		0 && 1 && 0&& 0 && 0 && 0 \\
		0 && 0 && 0  && 0 && 0 && 0 \\
		0       && 0        &&      0   && 0&& 0 && 0\\ 
		0       && 0        &&      0   && 0 &&   0&&0\\ 
		0&& 0 && 0  && 0  && 0 && 0
		\end{array}\right]}~, & 
	\e_3&: =& {\setlength{\arraycolsep}{2pt}\left[\begin{array}{rrrrrrrrrrc}
		0 && 0 && 0  && 0 && 0 && 0 \\
		0 && 0&& 0&& 0 && 0 && 0 \\
		0 && 0 && 1  && 0 && 0 && 0 \\
		0       && 0        &&      0   && 0&& 0 && 0\\ 
		0       && 0        &&      0   && 0 &&   0&&0\\ 
		0&& 0 && 0  && 0  && 0 && 0
		\end{array}\right]}~, &\\\\
	\e_4&: =& {\setlength{\arraycolsep}{2pt}\left[\begin{array}{rrrrrrrrrrc}
		0 && 0 && 0  && 0 && 0 && 0 \\
		0 && 0 && 0&& 0 && 0 && 0 \\
		0 && 0 && 0  && 0 && 0 && 0 \\
		0       && 0        &&      0   && 1&& 0 && 0\\ 
		0       && 0        &&      0   && 0 &&   0&&0\\ 
		0&& 0 && 0  && 0  && 0 && 0
		\end{array}\right]}~, &
	\e_5&: =& {\setlength{\arraycolsep}{2pt}\left[\begin{array}{rrrrrrrrrrc}
		0 && 0 && 0  && 0 && 0 && 0 \\
		0 && 0 && 0&& 0 && 0 && 0 \\
		0 && 0 && 0  && 0 && 0 && 0 \\
		0       && 0        &&      0   && 0&& 0 && 0\\ 
		0       && 0        &&      0   && 0 &&   1&&0\\ 
		0&& 0 && 0  && 0  && 0 && 0
		\end{array}\right]}~, &
	\e_6&: =& {\setlength{\arraycolsep}{2pt}\left[\begin{array}{rrrrrrrrrrc}
		0 && 0 && 0  && 0 && 0 && 0 \\
		0 && 0 && 0&& 0 && 0 && 0 \\
		0 && 0 && 0  && 0 && 0 && 0 \\
		0       && 0        &&      0   && 0&& 0 && 0\\ 
		0       && 0        &&      0   && 0 &&   0&&0\\ 
		0&& 0 && 0  && 0  && 0 && 1
		\end{array}\right]}~.\end{array}\]
\end{Remark}
\begin{proof}
	We have $\e_s\Lambda_{(3)} \e_s\cong R $ for $s\in[1,5]$. Therefore it follows that $\e_1,\e_2,\e_3,\e_4, \e_5$ are primitive.  
	
	To show that that $\e_6$ is primitive, we {\it claim} that the ring $\e_6\Lambda_{(3)}\e_6\cong R[\overline{\eta},\overline{\xi}]$ is local.
	
	We have $\UU( R[\overline{\eta},\overline{\xi}])= R[\overline{\eta},\overline{\xi}]\setminus (3,\overline{\eta},\overline{\xi})$ . In fact, for $u:= a + b\overline{\eta} +c\overline{\xi}$ with $a\in R\setminus(3)$ and $b,c\in R$, the inverse is given by $u^{-1}= a^{-1} - a^{-2}b\overline{\eta} -a^{-2}c\overline{\xi}$ as \begin{center} $uu^{-1}=aa^{-1} +(-a^{-1}b + a^{-1}b)\overline{\eta}+(-a^{-1}c + a^{-1}c)\overline{\xi}=1$ .\end{center}
	
	Thus the nonunits of $R[\overline{\eta},\overline{\xi}]$ form an ideal and so $R[\overline{\eta},\overline{\xi}]$ is a local ring. This proves the {\it claim}. 
\end{proof}

\subsection{\texorpdfstring{$\B_{\Z_{(3)}}(\sd,\sd)$}{B\_Z\_(3)(S\textthreeinferior,S\textthreeinferior)} and \texorpdfstring{$\B_{\F_3}(\sd,\sd)$}{B\_F\_3(S\textthreeinferior,S\textthreeinferior)} as path algebras modulo relations} Write $R :=\Z_{(3)}.$
We aim to write $\Lambda_{(3)}$, up to Morita equivalence, as path algebra modulo relations.\\ The $R$-algebra $\Lambda_{(3)}$ is Morita equivalent to $\Lambda'_{(3)} := (\e_3 +\e_4 +\e_5+\e_6)\Lambda_{(3)}(\e_3 +\e_4 +\e_5+\e_6)$ since \mbox{$\Lambda_{(3)}\e_1\cong \Lambda_{(3)}\e_2\cong\Lambda_{(3)}\e_3$} using multiplication with elements of $\Lambda_{(3)}$ with a single nonzero entry 1 in the upper $(3\ti 3)$-corner. We have the $R$-linear basis of $\Lambda'_{(3)}$ consisting of 

\[\scriptsize
\begin{array}{rclcrclcrclcrcl}
\e_3&: =& {\setlength{\arraycolsep}{2pt}\left[\begin{array}{rrrrrrrrrrc}
	0 && 0 && 0  && 0 && 0 && 0 \\
	0 && 0&& 0&& 0 && 0 && 0 \\
	0 && 0 && 1  && 0 && 0 && 0 \\
	0       && 0        &&      0   && 0&& 0 && 0\\ 
	0       && 0        &&      0   && 0 &&   0&&0\\ 
	0&& 0 && 0  && 0  && 0 && 0
	\end{array}\right]}~, &
\e_4&: =& {\setlength{\arraycolsep}{2pt}\left[\begin{array}{rrrrrrrrrrc}
	0 && 0 && 0  && 0 && 0 && 0 \\
	0 && 0 && 0&& 0 && 0 && 0 \\
	0 && 0 && 0  && 0 && 0 && 0 \\
	0       && 0        &&      0   && 1&& 0 && 0\\ 
	0       && 0        &&      0   && 0 &&   0&&0\\ 
	0&& 0 && 0  && 0  && 0 && 0
	\end{array}\right]}~, &
\e_5&: =& {\setlength{\arraycolsep}{2pt}\left[\begin{array}{rrrrrrrrrrc}
	0 && 0 && 0  && 0 && 0 && 0 \\
	0 && 0 && 0&& 0 && 0 && 0 \\
	0 && 0 && 0  && 0 && 0 && 0 \\
	0       && 0        &&      0   && 0&& 0 && 0\\ 
	0       && 0        &&      0   && 0 &&   1&&0\\ 
	0&& 0 && 0  && 0  && 0 && 0
	\end{array}\right]}~,\\\\
\e_6&: =& {\setlength{\arraycolsep}{2pt}\left[\begin{array}{rrrrrrrrrrc}
	0 && 0 && 0  && 0 && 0 && 0 \\
	0 && 0 && 0&& 0 && 0 && 0 \\
	0 && 0 && 0  && 0 && 0 && 0 \\
	0       && 0        &&      0   && 0&& 0 && 0\\ 
	0       && 0        &&      0   && 0 &&   0&&0\\ 
	0&& 0 && 0  && 0  && 0 && 1
	\end{array}\right]}~,&

\tau_1&: =&  {\setlength{\arraycolsep}{2pt}\left[\begin{array}{rrrrrrrrrrc}
	0 && 0 && 0  && 0 && 0 && 0 \\
	0 && 0 && 0&& 0 && 0 && 0 \\
	0 && 0 && 0  && 0 && 0 && 0 \\
	0       && 0        &&      0   && 0&& 0 && 0\\ 
	0       && 0        &&      0   && 0 &&   0&&0\\ 
	0&& 0 && 3  && 0  && 0 && 0
	\end{array}\right]}~, &
\tau_2&: =&{\setlength{\arraycolsep}{2pt}\left[\begin{array}{rrrrrrrrrrc}
	0 && 0 && 0  && 0 && 0 && 0 \\
	0 && 0 && 0&& 0 && 0 && 0 \\
	0 && 0 && 0  && 0 && 0 && 1 \\
	0       && 0        &&      0   && 0&& 0 && 0\\ 
	0       && 0        &&      0   && 0 &&   0&&0\\ 
	0&& 0 && 0  && 0  && 0 && 0
	\end{array}\right]}~,  \end{array}\]
\[\scriptsize\begin{array}{rclcrclcrclcrcl}
\tau_3&: =&{\setlength{\arraycolsep}{2pt}\left[\begin{array}{rrrrrrrrrrc}
	0 && 0 && 0  && 0 && 0 && 0 \\
	0 && 0 && 0&& 0 && 0 && 0 \\
	0 && 0 && 0  && 0 && 0 && 0 \\
	0       && 0        &&      0   && 0&& 0 && 0\\ 
	0       && 0        &&      0   && 0 &&   0&&0\\ 
	0&& 0 && 0  && 1  && 0 && 0
	\end{array}\right]}~, &
\tau_4&: =& {\setlength{\arraycolsep}{2pt}\left[\begin{array}{rrrrrrrrrrc}
	0 && 0 && 0  && 0 && 0 && 0 \\
	0 && 0 && 0&& 0 && 0 && 0 \\
	0 && 0 && 0  && 0 && 0 && 0 \\
	0       && 0        &&      0   && 0&& 0 && 1\\ 
	0       && 0        &&      0   && 0 &&   0&&0\\ 
	0&& 0 && 0  && 0  && 0 && 0
	\end{array}\right]}~, &
\tau_5&: =& {\setlength{\arraycolsep}{2pt}\left[\begin{array}{rrrrrrrrrrc}
	0 && 0 && 0  && 0 && 0 && 0 \\
	0 && 0 && 0&& 0 && 0 && 0 \\
	0 && 0 && 0  && 0 && 0 && 0 \\
	0       && 0        &&      0   && 0&& 0 && 0\\ 
	0       && 0        &&      0   && 0 &&   0&&0\\ 
	0&& 0 && 0  && 0  && 0 && 3\oeta
	\end{array}\right]}~, \\\\
\tau_6&: =&{\setlength{\arraycolsep}{2pt}\left[\begin{array}{rrrrrrrrrrc}
	0 && 0 && 0  && 0 && 0 && 0 \\
	0 && 0 && 0&& 0 && 0 && 0 \\
	0 && 0 && 0  && 0 && 0 && 0 \\
	0       && 0        &&      0   && 0&& 0 && 0\\ 
	0       && 0        &&      0   && 0 &&   0&&0\\ 
	0&& 0 && 0  && 0  && 0 && \oxi
	\end{array}\right]}~. 
\end{array}\]
We have $\tau_5=\tau_1\tau_2$ and $\tau_6= \tau_3\tau_4+4\tau_1\tau_2$ . Hence, as an $R$-algebra $\Lambda'_{(3)}$ is generated by $\e_3,\e_4,\e_5,\e_6,\tau_1,\tau_2,\tau_3,\tau_4$~.

Consider the quiver $\Psi:=\left[{\xymatrix@C-0.5em@R-0.5em{
		\tilde \e_5 & \tilde \e_3\ar@/^1pc/[rr]^-{\hspace{-0.7mm}\ttau_2} & &\tilde \e_6\ar@/^1pc/[ll]^-{\hspace{-0.7mm}\ttau_1}\ar@/_1pc/[rr]_-{\hspace{-0.7mm}\ttau_3} & &  \tilde \e_4\ar@/_1pc/[ll]_-{\hspace{-0.7mm}\ttau_4}
}}\right].$
We have a surjective $R$-algebra morphism $\phi: R\Psi\rightarrow \Lambda'_{(3)}$ by sending \[\begin{array}{rclccrclccrclccrclccrcl}
\tilde \e_3 &\mapsto& \e_3 &,&
\tilde \e_4 &\mapsto& \e_4 &,&
\tilde \e_5 &\mapsto& \e_5 &,&
\tilde \e_6 &\mapsto& \e_6 &,&\\
\ttau_1 &\mapsto&\tau_1  &,&
\ttau_2 &\mapsto&\tau_2  &,&
\ttau_3 &\mapsto&\tau_3  &,&
\ttau_4 &\mapsto&\tau_4  &.&
\end{array}\] 

We establish the following multiplication trees, where we underline the elements that are not in an $R$-linear relation with previous elements. 

The multiplication tree of the idempotent $\e_5$ consists only of the element $\e_5$.

\[{\xymatrix@C-0.5em@R-0.5em{
		\ul{\e_4}\ar[rr]_{\hspace{-0.5mm}\tau_4}& & \ul{\tau_4}\ar[d]_{\hspace{-0.5mm}\tau_1}\ar[rr]_{\hspace{-0.5mm}\tau_3} && \tau_4\tau_3=0\\ 
		&  &\tau_4\tau_1=0 & & & \\
}}{\xymatrix@C-0.5em@R-0.5em{
		\ul{\e_3}\ar[rr]_{\hspace{-0.5mm}\tau_2}& & \ul{\tau_2}\ar[d]_{\hspace{-0.5mm}\tau_3}\ar[rr]_{\hspace{-0.5mm}\tau_1} && \tau_2\tau_1=0\\ 
		&  &\tau_2\tau_3=0 & & & \\
}}\]

\[{\xymatrix@C-0.5em@R-0.5em{
		& & &    \ul{\e_6}\ar[dr]_{\hspace{-0.5mm}\tau_1}\ar[dl]_{\hspace{-0.5mm}\tau_3} & & & \\
		&\tau_3\tau_4\tau_3=0 & \ul{\tau_3}\ar[dl]_{\tau_4} & &\ul{\tau_1}\ar[r]^-{\hspace{-0.5mm}\tau_2} &  \ul{\tau_1\tau_2}\ar[d]_{\hspace{-0.5mm}\tau_3}\ar[rr]_{\hspace{-0.5mm}\tau_1} & &\tau_1\tau_2\tau_1=0 \\
		& \ul{\tau_3\tau_4}\ar[d]_{\hspace{-0.5mm}\tau_1}\ar[u]_{\hspace{-0.5mm}\tau_3}& & & &\tau_1\tau_2\tau_3=0  &\\
		& \tau_3\tau_4\tau_1=0 & & & & & & & & & \\
}}\]

So the  kernel of $\varphi$ contains the elements: $\ttau_4\ttau_1,~\ttau_4\ttau_3,~ \ttau_2\ttau_1,~ \ttau_2\ttau_3,~\ttau_3\ttau_4\ttau_3,~\ttau_2\ttau_3,~ \ttau_3\ttau_4 \ttau_1,~\ttau_1\ttau_2 \ttau_3,~\ttau_1\ttau_2 \ttau_1$.

Let $I$ be the ideal generated by these elements. So, $I\subseteq\kernel(\phi)$.  Therefore, $\phi$ induces a surjective $R$-algebra morphism from $R\Psi/I$ to $\Lambda_{(3)}'\ab$. We may reduce the list of generators to obtain $I = (\ttau_4\ttau_3,~\ttau_4\ttau_1,~\ttau_2\ttau_1,~\ttau_2\ttau_3 )$.

Note that $R\Psi / I$ is $R$-linearly generated by \begin{center} $\mathcal{N}:=\{\tilde \e_3+ I, \tilde \e_4+I,\tilde \e_5+I,\tilde \e_6+I, \ttau_1 +I, \ttau_2+I, \ttau_3+I,\ttau_4+I, \ttau_3\ttau_4+I, \ttau_1\ttau_2+I\}$,\end{center}cf.\ the underlined elements above.

To see that, note that a product $\xi$ of $k$ generators may be written as a product in $\mathcal{N}$ of $k'$ generators and a product of $k''$ generators, where $k=k'+k''$ and where $k'$ is choosen maximal. If $k''\geq 1$ then, using the trees above, we have $\xi=0$.  Moreover, note that $|\mathcal{N}|=10=\rk_R(\Lambda'_{(3)})$. 

Since we have an surjective algebra morphism from $R\Psi/I$ to $\Lambda'_{(3)}$, this rank argument shows this morphism to be bijective. In particular, $I=\kernel(\phi)$.

So, we obtain the

\begin{Proposition}\label{pfadsdZ3} Recall that $I= (\ttau_4\ttau_3,~\ttau_2\ttau_1,~\ttau_4\ttau_1,~\ttau_2\ttau_3 )$.
	
	We have the isomorphisms of $R$-algebras
	\[\begin{array}{rcl}
	\Lambda'_{(3)}&\xrightarrow{\sim}&R\left[{\xymatrix@C-0.5em@R-0.5em{
			\tilde \e_5 & \tilde \e_3\ar@/^2pc/[rr]^-{\hspace{-0.7mm}\ttau_2} & &\tilde \e_6\ar@/^2pc/[ll]^-{\hspace{-0.7mm}\ttau_1}\ar@/_2pc/[rr]_-{\hspace{-0.7mm}\ttau_3} & &  \tilde \e_4\ar@/_2pc/[ll]_-{\hspace{-0.7mm}\ttau_4}
	}}\right] /I\\
	\e_i&\mapsto&  \tilde \e_i +I \text{ for $i\in [3,6]$}  \\
	\tau_i&\mapsto& \ttau_i +I \text{ for $i\in [1,4]$} 
	\end{array}\]
	Recall that $\B_{\Z_{(3)}}(\sd,\sd)$ is Morita equivalent to $\Lambda'_{(3)}\ab$.
\end{Proposition}

\begin{Corollary}\label{pfadsdF3}
	As $\F_3$-algebras, we have\[\Lambda'_{(3)}/3\Lambda'_{(3)}~\cong~\F_3~\left[{\xymatrix@C-0.5em@R-0.5em{
			\tilde \e_5 & \tilde \e_3\ar@/^2pc/[rr]^-{\hspace{-0.7mm}\ttau_2} & &\tilde \e_6\ar@/^2pc/[ll]^-{\hspace{-0.7mm}\ttau_1}\ar@/_2pc/[rr]_-{\hspace{-0.7mm}\ttau_3} & &  \tilde \e_4\ar@/_2pc/[ll]_-{\hspace{-0.7mm}\ttau_4}
	}}\right] /(\ttau_4\ttau_3,~\ttau_4\ttau_1,~\ttau_2\ttau_1,~\ttau_2\ttau_3 )~.\]
	Recall that $\B_{\F_{3}}(\sd,\sd)$ is Morita equivalent to $\Lambda'_{(3)}/3\Lambda'_{(3)}\ab$.
\end{Corollary}

\vspace*{5mm}

\begin{flushleft}
Nora Krau\ss \\
Universit\"at Stuttgart \\
Fachbereich Mathematik \\
Pfaffenwaldring 57 \\ 
70569 Stuttgart \\
\verb|kraussna@mathematik.uni-stuttgart.de|
\end{flushleft}
\end{document}